\documentclass[reqno,twoside,a4paper]{amsart}
\makeatletter                                                  %
\def\th@plain{\slshape}                                        %
\makeatother                                                   %

\usepackage[italian,english]{babel} 
\usepackage[utf8x]{inputenc}
\usepackage[T1]{fontenc} 
\usepackage{indentfirst} 
\usepackage{amsmath,amsfonts,amssymb,amsthm}
\usepackage[all]{xy}
\usepackage{xcolor}
\usepackage{graphicx}
\usepackage{caption}
\usepackage{subcaption}
\usepackage[font={small}, labelfont=bf]{caption}[2004/07/16]
\usepackage{url,mathrsfs}
\usepackage{tikz}
\usepackage{layaureo}

\makeatletter
\def\paragraph{\@startsection{paragraph}{4}%
  \z@\z@{-\fontdimen2\font}%
  {\normalfont\bfseries}}
\makeatother

\numberwithin{equation}{section}
\newcommand*{\diff}{\mathop{}\!\mathrm{d}}

\theoremstyle{plain}
\newtheorem{thm}{Theorem}[section]
\newtheorem{prop}[thm]{Proposition}

\newtheorem{lemma}[thm]{Lemma}

\theoremstyle{definition}
\newtheorem{defin}[thm]{Definition}
\newtheorem{example}[thm]{Example}

\theoremstyle{remark}
\newtheorem{remark}{Remark}

\newcommand{\C}{\mathbb{C}}

\newcommand{\R}{\mathbb{R}}
\newcommand{\Z}{\mathbb{Z}}
\newcommand{\N}{\mathbb{N}}
\newcommand{\T}{\mathbb{T}}

\newcommand{\bigslant}[2]{{\raisebox{.2em}{$#1$}\left/\raisebox{-.2em}{$#2$}\right.}}

\newcommand{\norma}[1]{\lVert#1\rVert}
\newcommand{\mnorm}[1]{\left\vert\mkern-1mu\left\vert\mkern-1mu\left\vert #1 \right\vert\mkern-1mu\right\vert\mkern-1mu\right\vert}
\newcommand{\modulo}[1]{\left\lvert#1\right\rvert}
\newcommand{\newword}[1]{\textsl{#1}}

\newcommand{\clzn}{c_{l,n}}
\newcommand{\clznp}{c_{l,n}\rq{}}

\newcommand{\clznn}{c_{l,n,N}}

\newcommand{\ntyx}{\underline{n}_t(\widehat{\mathbf{x}})}
\newcommand{\ntjx}{\underline{n}_{t_j}(\widehat{\mathbf{x}})}

\newcommand{\ptx}[1]{\mathcal{P}_{#1}(t, \widehat{\mathbf{x}})}

\newcommand{\one}{{\rm 1\mskip-4mu l}}

\DeclareMathOperator{\misura}{Leb}

\DeclareMathOperator{\image}{Im}

\DeclareMathOperator{\Id}{Id}

\DeclareMathOperator{\abel}{ab}

\begin{document}

\bibliographystyle{plain}

\sloppy

\title[Mixing for suspension flows over skew-translations]{Mixing for suspension flows over skew-translations and time-changes of quasi-abelian filiform nilflows}

\author[D.~Ravotti]{Davide Ravotti}
\address{School of Mathematics\\
University of Bristol\\
University Walk\\
BS8 1TW Bristol, UK}
\email{davide.ravotti@bristol.ac.uk}

\begin{abstract}
We consider suspension flows over uniquely ergodic skew-translations on a $d$-dimensional torus $\T^d$, for $d \geq 2$. We prove that there exists a set $\mathscr{R}$ of smooth functions, which is dense in the space $\mathscr{C}(\T^d)$ of continuous functions, such that every roof function in $\mathscr{R}$ which is not cohomologous to a constant induces a mixing suspension flow. We also construct a dense set of mixing examples which is explicitly described in terms of their Fourier coefficients. In the language of nilflows on nilmanifolds, our result implies that, for every uniquely ergodic nilflow on a quasi-abelian filiform nilmanifold, there exists a dense subspace of smooth time-changes in which mixing occurs if and only if the time-change is not cohomologous to a constant. This generalizes a theorem by Avila, Forni and Ulcigrai (J.~Diff.~Geom., 2011) for the classical Heisenberg group.
\end{abstract}


\thanks{\emph{2010 Math.~Subj.~Class.}: 37A25, 37C40.}

\maketitle

\section{Introduction}

In this paper, we investigate the ergodic properties of a class of \newword{parabolic} flows, i.e.~flows for which the divergence of nearby points is polynomial in time. Examples of parabolic flows in the homogeneous set-up are given by the classical horocycle flow (more generally, by unipotent flows on semisimple Lie groups) and by nilflows on nilmanifolds.
A \newword{nilmanifold} $M$ is the quotient of a nilpotent Lie group $G$ by a lattice $\Lambda$, namely a discrete subgroup such that the quotient space $M=\Lambda \backslash G$ has finite Haar measure. A flow on $M$ induced by a left-invariant vector field is called a \newword{nilflow}. The study of nilflows has also applications to number theory, for example to the distribution of fractional parts of polynomials and to estimates of theta sums (see, e.g., \cite{furstenberg:unique, furstenberg:thetasums, flaminio:nilflows}). The ergodic properties of nilflows are well-understood: almost every nilflow, although it is uniquely ergodic, is not mixing; see \S\ref{s:2} for a detailed discussion.

Very little is known for general smooth parabolic flows, even for smooth perturbations of homogeneous ones. We consider a simple class of smooth perturbations given by performing a \newword{time-change}. Roughly speaking, a smooth time-change of a flow is obtained by keeping the orbits fixed and varying smoothly the speed of the points; a precise definition is given in \S\ref{s:2.1}. 
A natural question is which ergodic properties persist under perturbations: ergodicity is preserved by time-changes, but mixing is more delicate. The case of time-changes of the horocycle flow and of unipotent flows on semisimple Lie groups  have been studied by many authors, including Marcus \cite{marcus:horocycle}, Forni and Ulcigrai \cite{forni:horocycle}, Tiedra de Aldecoa \cite{tiedra:spectrum}, and Simonelli \cite{simonelli:acspectrum}; in this paper, building on a previous work by Avila, Forni and Ulcigrai \cite{avila:heisenberg}, we address the question of mixing for time-changes of nilflows.  

The simplest non-abelian nilpotent group is the Heisenberg group $H$ consisting of $3 \times 3$ upper triangular unipotent matrices, which is 3-dimensional and 2-step nilpotent. Let $\{\varphi_t\}_{t \in \R}$ be a uniquely ergodic nilflow on a nilmanifold $M = \Lambda \backslash H$. There exists a cross-section $\Sigma \subset M$ isomorphic to the 2-dimensional torus $\T^2$ such that the Poincaré map $T \colon \T^2 \to \T^2$ is a uniquely ergodic skew-translation of the form $T(x,y) = (x+ \alpha, x + y + \beta)$, for some $\alpha, \beta \in \R$, and $\{\varphi_t\}_{t \in \R}$ is isomorphic to the suspension flow over $(\T^2, T)$ with a constant roof function. 

Performing a time-change of the nilflow $\{\varphi_t\}_{t \in \R}$ amounts to changing the roof function. Clearly, constant time-changes do not produce mixing, therefore we want to rule out from our analysis all time-changes which behave like constants. This class is given by time-changes \newword{cohomologous to a constant}, see \S\ref{s:introsf}. Avila, Forni and Ulcigrai in \cite{avila:heisenberg} proved that there exists a set $\mathscr{R}$ of smooth functions which is dense in $\mathscr{C}(\T^2)$ such that for all positive $\Psi \in \mathscr{R}$, the suspension flow over $(\T^2,T)$ with roof function $\Psi$ is mixing if and only if $\Psi$ is not cohomologous to a constant. Moreover, they showed that this condition can be checked explicitly. They also prove an analogous result for smooth time-changes of the original Heisenberg nilflow. In this paper, we generalize these results to higher dimensions, see Theorem \ref{th:intro1} and \ref{th:intro2}.


\subsection{Suspension flows over skew-translations}\label{s:introsf}

Let us consider a $d$-dimensional torus $\T^d$ for some $d \geq 2$. We denote points in $\T^d$ by row vectors $\mathbf{x} = (x_1, \dots, x_d)$. Let $\misura_d$ be the $d$-dimensional Lebesgue measure. Let $T \colon \T^d \to \T^d$ be a skew-translation of the form $T \mathbf{x} = \mathbf{x} A + \mathbf{b}$, where $A= (a_{i,j})_{1 \leq i,j \leq d}$ is a $d \times d$ upper-triangular unipotent matrix with integer coefficients such that $A \neq \Id$ and $\mathbf{b}=(b_1, \dots, b_d) \in \T^d$; namely, for $\mathbf{x} = (x_1, \dots, x_d) \in \T^d$, define
\begin{equation}\label{eq:dot}
T(x_1, \dots, x_d) = (x_1, \dots, x_d)
\begin{pmatrix}
1 & a_{1,2} & \cdots & a_{1,d} \\
\ & 1 & \ddots & \vdots \\
\ & \  & \ddots & a_{d-1,d}\\
\ & \ & \ & 1
\end{pmatrix} 
+ (b_1, \dots, b_d).
\end{equation}
We suppose that the skew-translation $T \colon \T^d \to \T^d$ is ergodic (equivalently, uniquely ergodic~\cite{furstenberg:unique}). 

Any positive continuous function $\Psi \colon \T^d \to \R_{>0}$ defines a \newword{suspension flow} in the following way: consider the equivalence relation $\sim_{\Psi}$ on $\T^d \times \R$ generated by all the pairs $\{ (\mathbf{x} , \Psi(\mathbf{x} )), (T\mathbf{x} ,0) \}$. 
The suspension flow $\{T^{\Psi}_t\}_{t \in \R}$ on the space $\bigslant{\T^d \times \R}{\sim}_{\Psi}$ is defined by moving at unit speed in the vertical direction for time $t$.
If we denote by $S_n(\Psi)$ the $n$-th Birkhoff sum along the orbits of $T$, namely
\begin{equation}\label{eq:birksum}
S_n(\Psi)(\mathbf{x} ) := 
\begin{cases}
\sum_{i=0}^{n-1} \Psi \circ T^i \mathbf{x} & \text{ if } n>0,\\
0 & \text{ if } n=0,\\
-\sum_{i=-1}^{n} \Psi \circ T^i \mathbf{x} & \text{ if } n<0,
\end{cases}
\end{equation}
then we can write an explicit formula for the suspension flow
\begin{equation}\label{eq:sflow}
T^{\Psi}_t(\mathbf{x} ,r) = \left( T^{n_t(\mathbf{x} ,r)}\mathbf{x} , r+t-S_{n_t(\mathbf{x} ,r)}(\Psi)(\mathbf{x} ) \right),
\end{equation}
where $n_t(\mathbf{x} ,r)$ denotes the maximum $n \geq 0$ such that $S_n(\Psi)(\mathbf{x} ) \leq r+t$.

For every roof function $\Psi$, unique ergodicity of $\{T^{\Psi}_t\}_{t \in \R}$ is equivalent to unique ergodicity of $T$; on the other hand, mixing for the suspension flow is a delicate question. 
Suppose that there exists a measurable (resp.~smooth) function $u \colon \T^d \to \R$ such that 
$$
\Psi - \int_{\T^d} \Psi \diff \misura_d=u \circ T - u.
$$
We say that $u \circ T - u$ is a \newword{measurable} (resp.~\newword{smooth}) \newword{coboundary for $T$} and $\Psi$ and the constant $\int \Psi$ are \newword{measurably} (resp.~\newword{smoothly}) \newword{cohomologous w.r.t.~$T$}. 
In this case, the map $\zeta \colon (\mathbf{x} ,r) \mapsto (\mathbf{x} ,r+u(\mathbf{x} ))$ defined on $\T^d \times \R$ descends to the quotient and maps the invariant measure for $\{T^{\Psi}_t\}_{t \in \R}$ to the invariant measure of the suspension flow with constant roof function $\int \Psi$.
In particular, any roof function $\Psi$ cohomologous to a constant induces a non-mixing flow. Our main result, Theorem \ref{th:intro1} below, shows that, within a dense subspace $\mathscr{R}$, the condition of not being cohomologous to a constant is also sufficient for mixing of the suspension flow.

Determining whether a function is a measurable coboundary is not, in general, effectively possible. An exception is the case of a 2-dimensional skew-translation treated in \cite{avila:heisenberg}, where measurable coboundaries are explicitly characterized in terms of invariant distributions for the Heisenberg nilflow. At present, this result appears not to be generalizable to higher dimensions, since it relies on sharp estimates on Weyl sums (see \cite{flaminio:nilflows} and references therein), which are available only for degree two. However, exploiting the 2-dimensional case, we construct a dense and explicitly described set of mixing examples for a large class of higher dimensional skew-traslations, which includes the ones arising from filiform nilflows, see \S\ref{s:2.1}.
 
\begin{thm}\label{th:intro1}
\begin{itemize}
\item[(a)] There exists a subspace $\mathscr{R}$ of smooth functions, which is dense in $\mathscr{C}(\T^d)$ w.r.t.~$\norma{\cdot}_{\infty}$, such that for all positive $\Psi \in \mathscr{R}$ the suspension flow over $(\T^d, T)$ with roof function $\Psi$ is mixing if and only if $\Psi$ is not cohomologous to a constant. 
\item[(b)] If the entries above the diagonal are non-zero, namely if $a_{i,i+1} \neq 0$ for $i=1, \dots, d-1$ in \eqref{eq:dot}, then there exists a dense set $\mathscr{M}$ of mixing examples which is explicitly described in terms of their Fourier coefficients.
\end{itemize}
\end{thm}


\subsection{Time-changes of quasi-abelian filiform nilflows}\label{s:2.1}

From Theorem \ref{th:intro1} we deduce an analogous statement for time-changes of \newword{quasi-abelian filiform} nilflows, which are nilflows on a class of higher dimensional and higher step nilpotent groups. We recall all the relevant definitions, referring the reader to \cite[\S1]{gallot:riemannian} and to \cite{corwin:nilpotent} for an introduction to homogeneous flows on Lie groups and for a background on nilpotent groups respectively.

Let $G$ be a connected, simply connected Lie group and let $\mathfrak{g}$ be its Lie algebra. For any vector $\mathbf{w} \in \mathfrak{g}$ it is possible to define a vector field $W$ over $G$ in the following way: if $L_g \colon h \mapsto gh$ denotes the left-multiplication by $g \in G$, we set $W_g = (L_g)_{\ast} \mathbf{w}$, where $(L_g)_{\ast} = (L_g)_{\ast}(\Id) \colon \mathfrak{g} \to T_gG$ is the differential of $L_g$ at the identity.
The vector field $W$ is left-invariant, i.e.~$(L_g)_{\ast}W = W$ for all $g \in G$ and indeed the map $\mathbf{w} \mapsto W$ is a bijection between $\mathfrak{g}$ and the set of left-invariant vector fields over $G$ \cite[Proposition 1.72]{gallot:riemannian}. By a little abuse of notation, we will identify $\mathbf{w}$ with $W$.

The flow $\{\varphi^{\mathbf{w}}_t\}_{t \in \R}$ defined by $\mathbf{w}$ is explicitly given by $\varphi^{\mathbf{w}}_t(g) = g \exp(t\mathbf{w})$.
Given a lattice $\Lambda \leq G$, the flow $\{\varphi^{\mathbf{w}}_t\}_{t \in \R}$ induces a flow on the quotient $\Lambda \backslash G$; moreover, since $G$ is unimodular \cite[Proposition 9.20]{einsiedler:ergodic}, 
$\{\varphi^{\mathbf{w}}_t\}_{t \in \R}$ preserves both the Haar measure $\mu$ on $G$ and its push-forward onto the quotient $\Lambda \backslash G$. We will denote the latter with $\mu$ as well and we will assume it is appropriately normalized, namely $\mu (\Lambda \backslash G)=1$.

Denote by $[\ , \ ]_G$ the commutator in $G$. We recall that a group $G$ is $n$-step nilpotent if $G^{(n+1)} = \{\Id\}$ and $G^{(n)} \neq \{\Id\}$, where $G^{(1)} = G$ and $G^{(i+1)} = [G,G^{(i)}]_G$.
Setting $\mathfrak{g}^{(1)} = \mathfrak{g}$ and $\mathfrak{g}^{(i+1)} = [\mathfrak{g},\mathfrak{g}^{(i)}]$, where $[\ ,\ ]$ denotes the Lie bracket in $\mathfrak{g}$, it is a well-known fact that $\mathfrak{g}^{(i)}$ is the Lie algebra of $G^{(i)}$. In particular, $\mathfrak{g}^{(n+1)} = \{ 0 \}$ and $\mathfrak{g}^{(n)} \neq \{0\}$; we say that $\mathfrak{g}$ is a $n$-step nilpotent Lie algebra.
Notice that in a $n$-step nilpotent algebra the centre is always nontrivial, more precisely $\mathfrak{g}^{(n)} \subseteq \mathfrak{z}(\mathfrak{g})$.

In this paper, we focus our attention on \newword{quasi-abelian filiform} groups $F_d$, which we introduce through their Lie algebras. 
The quasi-abelian filiform algebra $\mathfrak{f}_d$ of $F_d$ is the $(d+1)$-dimensional nilpotent Lie algebra spanned by $\mathcal{F}_d=\{\mathbf{f}_0, \dots, \mathbf{f}_{d} \}$ such that the only nontrivial brackets are $[\mathbf{f}_0, \mathbf{f}_i] = \mathbf{f}_{i+1}$ for $1 \leq i \leq d-1$. Then, $\mathfrak{f}_d$ is $d$-step nilpotent and we can represent it as a matrix algebra as
$$
x\mathbf{f}_0+\sum_{i=1}^d y_i \mathbf{f}_{i} \mapsto
\begin{pmatrix}
0 & x & \ & \ & y_d \\
\ & 0 & \ddots & \ & \vdots\\
\ & \  & \ddots & x & y_2\\
\ & \ & \ & 0 & y_1\\
\ & \ & \ & \ & 0
\end{pmatrix}.
$$
We remark that $F_1 \simeq \R^2$ and $F_2$ is the Heisenberg group $H$.

Let $F=F_d$ be a quasi-abelian filiform group and let $\Lambda < F$ be a lattice; the quotient $M=\Lambda \backslash F$ is said to be a \newword{quasi-abelian filiform nilmanifold} and every flow $\{\varphi^{\mathbf{w}}_t\}_{t \in \R}$ as above is called \newword{quasi-abelian filiform nilflow}. Almost every quasi-abelian filiform nilflow is uniquely ergodic but not weak mixing, see \S\ref{s:2}.

We say that a flow $\{ \varphi^{\alpha}_t \}_{t \in \R}$ on $M$ is a \newword{continuous} (resp.~\newword{smooth}) \newword{time-change of} $\{{\varphi}_t\}_{t \in \R}$ if there exists a continuous (resp.~smooth) function $\alpha \colon M \to \R_{>0}$ such that $\{ \varphi^{\alpha}_t \}_{t \in \R}$ is induced by $\alpha \mathbf{w}$ (recall that $\mathbf{w}$ is the vector field inducing $\{\varphi_t\}_{t \in \R}$). We say that $\alpha$ is the \newword{infinitesimal generator} of the time-change. The time-change given by $\alpha$ is cohomologous to a constant if there exists $\beta \colon M \to \R$ such that 
$$
\alpha - \int_M \alpha \diff \mu = \mathbf{w} \beta.
$$
As for suspension flows, time-changes cohomologous to constants are not mixing.
From Theorem \ref{th:intro1} we deduce the following result.

\begin{thm}\label{th:intro2}
Let $M = \Lambda \backslash F_d$ be a quasi-abelian filiform nilmanifold for some $d \geq 2$ and consider a uniquely ergodic quasi-abelian filiform nilflow on $M$. There exists a set of smooth time-changes, which is dense in the set of continuous time-changes, such that every element is mixing if and only if it is not cohomologous to a constant. 

Moreover, the set of mixing time-changes is dense in the set of continuous time-changes.
\end{thm}


\subsection{Outline of the paper.} 
Section \ref{section2} is devoted to explain the general strategy of the proof of Theorem \ref{th:intro1}-(a). 
First, we present a general mechanism that will allow us to reduce to consider a factor of the suspension flow for which the divergence of nearby points is of strictly higher order in the $x_d$-direction (we remark that $T$ acts as a translation in this latter coordinate). This is obtained by applying inductively Proposition \ref{th:induct}, whose proof is contained in \S\ref{proofofwrap}. Then, we prove mixing for the new suspension flow by showing that there is stretch of Birkhoff sums of the roof function $\Psi$ (see Theorem \ref{th:stretchingroof} in \S\ref{ST1}) and then using this stretch to show that segments in the $x_d$-direction get \newword{sheared} along the flow direction, as we explain in \S\ref{SU}. This is a well-known mechanism, often exploited in the context of parabolic flows to prove mixing. For example, it has been used by Marcus for time-changes of horocycle flows \cite{marcus:horocycle}, by Sinai and Khanin \cite{sinai:mixing}, Fayad \cite{fayad:mixing}, Ulcigrai \cite{ulcigrai:mixing} and the author \cite{ravotti:lhf} for suspension flows over rotations and interval exchange transformations. In our case, shearing comes from the fact that the roof function is not cohomologous to a constant and a decoupling argument, which generalizes the one used by Avila, Forni and Ulcigrai in \cite{avila:heisenberg} (although, in our higher dimensional setting, an additional geometric localization argument is needed).
In \S\ref{S:dom}, we use Proposition \ref{th:induct} to construct a dense set of mixing examples on any dimension, starting from the 2-dimensional ones, hence proving Theorem \ref{th:intro1}-(b). These roof functions are explicitly characterized in terms of their Fourier coefficients (see Lemma \ref{th:explicitcob}, which generalizes a result by Katok \cite[Theorem 11.25]{katok:combinatorial}).
Finally, in Section \ref{s:proof2}, we prove Theorem \ref{th:intro2} by constructing a cross-section for the quasi-abelian filiform nilflow such that, in appropriate coordinates, the Poincaré map is a skew-translation on $\T^d$, hence reducing the problem of time-changes of quasi-abelian filiform nilflows to the setting of Theorem \ref{th:intro1}.


\subsection*{Acknowledgments.} 
I would like to thank Giovanni Forni and my supervisor Corinna Ulcigrai for several useful discussions. I also thank the referee for his/her attentive reading and helpful comments on a previous version of the paper. The research leading to these results has received funding from the European Research Council under the European Union Seventh Framework Programme (FP/2007-2013)~/~ERC Grant Agreement n.~335989.


\section{Proof of Theorem \ref{th:intro1}-part a}\label{section2}

In this section, we present the general structure of the proof of Theorem \ref{th:intro1}-(a), stating some intermediate results, which are proved in later sections.

Let $T$ be a uniquely ergodic skew-translation as in \eqref{eq:dot}. 
If we denote by $E_j$ the image of the linear map $(A - \Id)^j$, we have a filtration of $\R^d$ into rational subspaces
$$
\R^d = E_0 > E_1 > \cdots > E_k > E_{k+1} = \{0\}.
$$
Up to a linear isomorphism, we can assume that the basis $\{ \mathbf{e}_1, \dots, \mathbf{e}_d \}$ of $\Z^d$ is adapted to the filtration above, in particular $\{\mathbf{e}_{d_0 +1}, \dots, \mathbf{e}_d \}$ is a basis of $E_k$, where $d-d_0 = \dim E_k $. 
Since $T$ is not a rotation, $k \geq 1$ and $1 \leq d_0 \leq d-1$. We remark that $\mathbf{w} \in E_j$ for $j \geq 1$ if and only if there exists $\mathbf{v} \in E_{j-1}$ such that $\mathbf{v}(A-\Id) = \mathbf{w}$, i.e.~$\mathbf{v}A = \mathbf{v}+\mathbf{w}$. In particular, for the basis elements $\mathbf{e}_{d_0+i} \in E_k \cap \Z^d$, for $i = 1, \dots, d-d_0$, there exists $\mathbf{v} \in E_{k-1} \cap \Z^d$ such that $\mathbf{v} A = \mathbf{v} + a \mathbf{e}_{d_0+i}$, for some $a \neq 0$.

We want to reduce to the case $d_0 = d-1$, that is $\dim E_k = 1$. In \S\ref{sectionwrap} we describe a general mechanism that allows us to deduce mixing from the assumption that a system with one less dimension is mixing. This motivates also the definition of the set $\mathscr{R}$, which is explained in \S\ref{R}. In \S\ref{dns} we prove $\mathscr{R}$ is dense in $\mathscr{C}(\T^d)$. Finally, in \S\ref{S:case2} and \S\ref{S:case1} we apply inductively the result of \S\ref{sectionwrap} to reduce to the case $d_0 = d-1$ and then we conclude the proof of Theorem \ref{th:intro1}-(a).


\subsection{The wrapping mechanism}\label{sectionwrap}
Let $\pi \colon \T^d \to \T^{d-1}$ be the projection given by suppressing the $d$-th coordinate. Then $\pi$ gives a factor of $(\T^d, T)$; more precisely, let $\widehat{A}=(a_{i,j})_{1 \leq i,j \leq d-1}$ be the $(d-1)\times(d-1)$ matrix obtained by removing the last row and the last column from $A$ and let $\widehat{\mathbf{b}} = \pi(\mathbf{b}) \in  \T^{d-1}$. Then, the skew-translation $\widehat{T} \colon \T^{d-1} \to \T^{d-1}$ defined by $\widehat{T} \mathbf{y} =\mathbf{y}\widehat{A} + \widehat{\mathbf{b}}$ makes the diagram
\begin{displaymath}
    \xymatrix{\T^d \ar[d]_{\pi} \ar[r]^{T} & \T^d\ar[d]^{\pi} \\
              \T^{d-1} \ar[r]^{\widehat{T}} & \T^{d-1}}
\end{displaymath}
commute.

Let us denote also by $\pi \colon \T^d \times \R \to \T^{d-1} \times \R$ the projection $\pi(\mathbf{x},r) = (\pi(\mathbf{x}),r)$. 
Let $\psi \colon \T^{d-1} \to \R_{>0}$ be a smooth function over $\T^{d-1}$ and consider the roof function $\psi \circ \pi$ over $(\T^d,T)$ which is constant in the $d$-th coordinate.
Then, $\pi$ is a factor map of the suspension flow $\{T_t^{\psi \circ \pi}\}_{t \in \R}$, namely
\begin{equation}\label{eq:factor}
(\pi \circ T_t^{\psi \circ \pi})(\mathbf{x}, r) = (\widehat{T}_t^{\psi} \circ \pi)(\mathbf{x}, r).
\end{equation}

As we discussed at the beginning of the section, there exists $\mathbf{v} \in \Z^d$ such that $\mathbf{v}A = \mathbf{v} + a\mathbf{e}_d$ for some $a \neq 0$.
This means that the images of segments parallel to $\mathbf{v}$ under $T$ get sheared in direction $\mathbf{e}_d$ and wrap around the circles parallel to $\mathbf{e}_d$.
Exploiting this shearing effect along the fibers of the projection $\pi$, it is possible to \lq\lq{}lift\rq\rq{}\ mixing from the quotient to the original suspension flow, namely the following result.
\begin{prop}\label{th:induct}
Let $\pi$ be the projection onto the first $d-1$ coordinates and let $\widehat{T} \colon \T^{d-1} \to \T^{d-1}$ be the corresponding factor. Let $\psi \colon \T^{d-1} \to \R_{>0}$ be a positive smooth function. If there exists $\mathbf{v} \in \Z^d$ such that $\mathbf{v}A = \mathbf{v} + a\mathbf{e}_d$ for some $a \neq 0$, then the suspension flow $\{ T_t^{\psi \circ \pi}\}_{t \in \R}$ over $(\T^d,T)$ is mixing if and only if the suspension flow $\{ \widehat{T}_t^{\psi} \}_{t \in \R}$ over $(\T^{d-1},\widehat{T})$ is mixing.
\end{prop}
The proof of Proposition \ref{th:induct} is presented in \S\ref{proofofwrap}.


\subsection{Definition of $\mathscr{R}$}\label{R}

Generalizing the notation of \S\ref{sectionwrap}, for each $i=1, \dots, d-1$, denote by $\pi_i \colon \T^d \to \T^i$ the projection onto the first $i$ coordinates and by $T_i \colon \T^i \to \T^i$ the corresponding factor map $\pi_i \circ T = T_i \circ \pi_i$. Let $\mathscr{P}(d)$ be the space of trigonometric polynomials over $\T^d$. For any $\Psi \in \mathscr{P}(d)$, we can write
$$
\Psi = \psi_{d-1} \circ \pi_{d-1} + \Psi_d^{\perp},
$$
where
$$
\psi_{d-1} ( \pi_{d-1} (\mathbf{x})) = \int_0^1 \Psi(\mathbf{x}) \diff x_d \text{\ \ \ and\ \ \ } \Psi_d^{\perp}(\mathbf{x}) = \Psi(\mathbf{x}) - \psi_{d-1}(\pi_{d-1} (\mathbf{x}) ).
$$
The function $\psi_{d-1} \circ \pi_{d-1}$ does not depend on the $x_d$-coordinate, thus we can see $\psi_{d-1}$ as a trigonometric polynomial over $\T^{d-1}$. Inductively, we write
\begin{equation}\label{eq:dec}
\Psi = \psi_{d_0} \circ \pi_{d_0} + \Psi_{d_0+1}^{\perp} \circ \pi_{d_0+1} + \cdots + \Psi_d^{\perp},
\end{equation}
where
$$
\psi_i \circ \pi_i = \int_0^1 \psi_{i+1} \circ \pi_{i+1} \diff x_{i+1} \text{\ \ \ and\ \ \ } \Psi_i^{\perp} \circ \pi_i = \psi_{i+1} \circ \pi_{i+1} - \psi_i \circ \pi_i.
$$
The integral of $\Psi_i^{\perp}$ in $\diff x_i$ is equal to zero, hence we have the decomposition
\begin{equation}\label{eq:directsum}
\mathscr{P}(d) = \mathscr{P}(d_0) \oplus \bigoplus_{i=d_0+1}^{d} \mathscr{Q}(i), \text{\ \ \ where\ \ \ } \mathscr{Q}(i) = \left\{ \Psi \in \mathscr{P}(i) : \int_0^1 \Psi \diff x_{i} \equiv 0 \right\}.
\end{equation}
Explicitly, let $e(x) = \exp(2 \pi i x)$ and consider a trigonometric polynomial of degree $m$,
$$
\Psi (\mathbf{x}) = \sum_{\mathbf{l} \in [-m,m]^{d} \cap \Z^{d}} c_{\mathbf{l}} e(\mathbf{l} \cdot \mathbf{x}) \in \mathscr{P}(d).
$$
Then, denoting $\mathbf{x}_i = \pi_i (\mathbf{x})$, we have
\begin{equation*}
\begin{split}
& \psi_{d_0}(\mathbf{x}_{d_0}) = \sum_{\mathbf{l}_{d_0} \in [-m,m]^{d_0} \cap \Z^{d_0}} c_{(l_1, \dots,l_{d_0}, 0, \dots, 0)} e(\mathbf{l}_{d_0} \cdot \mathbf{x}_{d_0}) \text{\ \ \ and}\\ 
&\Psi^{\perp}_{i} (\mathbf{x}_i) = \sum_{\mathbf{l}_i \in [-m,m]^{i} \cap \Z^{i},\ l_i \neq 0} c_{(l_1, \dots, l_i, 0, \dots, 0)} e(\mathbf{l}_i \cdot \mathbf{x}_i),
\end{split}
\end{equation*}
where the last sum is taken over all integer vectors $\mathbf{l}_i = \pi_i(\mathbf{l}) \in [-m,m]^{i} \cap \Z^{i}$ such that the last component $l_i \neq 0$.
\begin{defin}\label{def:R}
For each $\Psi \in \mathscr{P}(d)$ consider the decomposition \eqref{eq:dec}. We define the set $\mathscr{R}=\mathscr{R}(T) \subset \mathscr{P}(d)$ associated to the skew-translation $T$ by
\begin{equation*}
\begin{split}
\Psi \in \mathscr{R} \text{\ \ iff\ \ } & \Psi_{i}^{\perp} \text{ is a measurable coboundary for $T_{i}$ for all $i= d_0+2, \dots, d$}\\
& \text{and } \psi_{d_0} \text{ is smoothly cohomologous to a constant w.r.t.~$T_{d_0}$}.
\end{split}
\end{equation*}
\end{defin}


\subsection{Density}\label{dns}

We now prove that $\mathscr{R}$ is dense in $\mathscr{C}(\T^d)$ w.r.t.~$\norma{\cdot}_{\infty}$. By \eqref{eq:directsum}, we have to show that the set of trigonometric polynomials which are smoothly cohomologous to a constant w.r.t.~$T_{d_0}$ is dense in $\mathscr{P}(d_0)$ and that the set of measurable coboundaries for $T_{i}$ in $\mathscr{Q}(i)$ is dense in $\mathscr{Q}(i)$ for all $i = d_0+2, \dots, d$.
All factors $T_{d_0}, \dots, T_{d-1}$ are uniquely ergodic skew-translations of the same form as $T$, hence it suffices to prove the following lemma; the proof follows the same ideas as a result by Katok \cite[Proposition 10.13]{katok:combinatorial}.
\begin{lemma}\label{th:densecob}
We have the following.
\begin{itemize}
\item[(i)] The set of trigonometric polynomials which are smoothly cohomologous to a constant w.r.t.~$T$ is dense in $\mathscr{P}(d)$.
\item[(ii)] The set of smooth coboundaries for $T$ in $\mathscr{Q}(d)$ is dense in $\mathscr{Q}(d)$.
\end{itemize}
\end{lemma}
\begin{proof}
We show (ii); the proof of (i) is analogous. Define $P \colon \mathscr{Q}(d) \to \mathscr{Q}(d)$ by $P \Psi_d^{\perp} = \Psi_d^{\perp} \circ T - \Psi_d^{\perp}$; it is sufficient to show that $\mathscr{Q}(d) \subseteq \overline{\image P}$, where the closure is w.r.t.~$\norma{\cdot}_{\infty}$ in $\mathscr{Q}(d)$.

Suppose, by contradiction, that there exists $\Phi \in \mathscr{Q}(d)$ and $\Phi \notin \overline{\image P}$. By Hahn-Banach Theorem, there exists $\nu \colon \mathscr{Q}(d) \to \R$ linear and continuous such that $\nu(\Phi) =1$ and $\nu |_{\overline{\image P}} = 0$. 
We extend $\nu$ to a functional $\widetilde{\nu}$ on all $\mathscr{P}(d) = \mathscr{P}(d-1) \oplus \mathscr{Q}(d)$ by defining 
$$
\widetilde{\nu}(\psi_{d-1} \circ \pi_{d-1}+ \Psi_d^{\perp}) = \int_{\T^{d-1}} \psi_{d-1} \diff \misura_{d-1}+ \nu(\Psi_d^{\perp}). 
$$
It is easy to check that $\widetilde{\nu}$ is again linear and continuous, hence it uniquely defines a measure on $\T^{d}$.
For every $\Psi_d^{\perp} \in \mathscr{Q}(d)$ we have
$$
0 = \nu(P\Psi_d^{\perp}) = \nu( \Psi_d^{\perp} \circ T ) - \nu( \Psi_d^{\perp}),
$$
i.e.,~$\nu$ is $T$-invariant over $\mathscr{Q}(d)$. Therefore, for any $\Psi =\psi_{d-1} \circ \pi_{d-1} + \Psi_d^{\perp}\in \mathscr{P}(d)$, 
$$
\widetilde{\nu}(\Psi \circ T) = \int_{\T^{d-1}} \psi_{d-1} \circ T_{d-1}\ \diff \misura_{d-1} + \nu(\Psi_d^{\perp} \circ T ) = \int_{\T^{d-1}} \psi_{d-1} \ \diff \misura_{d-1} + \nu(\Psi_d^{\perp} )=\widetilde{\nu}(\Psi). 
$$
By unique ergodicity of $T$, we deduce that $\widetilde{\nu}=\misura_d$. We conclude 
$$
\widetilde{\nu}(\Phi) = \int_{\T^d} \Phi \diff \misura_d = 0, 
$$
in contradiction with $\nu(\Phi)=1$.
\end{proof}


\subsection{Proof of Theorem \ref{th:intro1}-(a): step 1}\label{S:case2}

Using Proposition \ref{th:induct}, we explain how to reduce the problem to the case of $\dim E_k =1$, where, we recall, $E_k$ is the image of $(A-\Id)^k$ and $(A-\Id)^{k+1}=0$. 
Let $\Psi \in \mathscr{R}$, and assume that it is not cohomologous to a constant w.r.t.~$T$. If $d_0 \leq d-2$, then, by definition of $\mathscr{R}$, the function $\Psi_d^{\perp}$ is a measurable coboundary for $T$, i.e.~$\Psi_d^{\perp} = u \circ T - u$ for some measurable function $u \colon \T^d \to \R$. 
We claim that $\psi_{d-1}$ is not cohomologous to a constant w.r.t.~the factor map $T_{d-1}$. By contradiction, suppose that $\psi_{d-1} - \int \psi_{d-1} = v \circ T_{d-1} - v$ for some $v \colon \T^{d-1} \to \R$. Then,
\begin{multline*}
\Psi - \int_{\T^d} \Psi \diff \misura_d = \psi_{d-1} \circ \pi_{d-1} + \Psi_{d}^{\perp} - \int_{\T^{d-1}} \psi_{d-1} \diff \misura_{d-1} \\
=  v \circ T_{d-1} \circ \pi_{d-1} - v \circ \pi_{d-1} + u \circ T - u = (v \circ \pi_{d-1} + u) \circ T - (v \circ \pi_{d-1} + u),
\end{multline*}
in contradiction with the assumption on $\Psi$.

\begin{remark}\label{rk:mixcob}
The differential of the map $\zeta \colon (\mathbf{x} ,r) \mapsto (\mathbf{x} ,r+u(\mathbf{x} ))$ defined on $\T^d \times \R$ has determinant 1, hence $\zeta$ preserves the $(d+1)$-dimensional Lebesgue measure. Moreover, one can check that $\zeta(\mathbf{x}, r + \Psi(\mathbf{x})) = \zeta(T\mathbf{x}, r)$; in particular, $\zeta$ descends to the quotient spaces $\zeta \colon \bigslant{\T^d \times \R}{\sim}_{\Psi} \to \bigslant{\T^d \times \R}{\sim}_{\psi_{d-1} \circ \pi_{d-1}}$ and hence maps the invariant measure for the suspension flow over $(\T^d, T)$ with roof function $\Psi = \psi_{d-1} \circ \pi_{d-1} + \Psi_d^{\perp}$ to the invariant measure of the one with roof function $\psi_{d-1} \circ \pi_{d-1}$. In particular, the former is mixing if and only if the latter is mixing.
\end{remark}

By Remark \ref{rk:mixcob} and by Proposition \ref{th:induct}, mixing of $\{ T_t^{\Psi} \}_{t \in \R}$ is equivalent to mixing of $\{ (T_{d-1})_t^{\psi_{d-1}} \}_{t \in \R}$, where $\psi_{d-1} \in \mathscr{R}(T_{d-1})$. Iterating this process for all $\Psi_i^{\perp}$ for $i = d_0+2, \dots, d$, we reduce to prove mixing for the suspension flow $\{ (T_{d_0+1})_t^{\psi_{d_0+1}} \}_{t \in \R}$ over $(\T^{d_0+1}, T_{d_0 +1})$ with roof function $\psi_{d_0 +1} \in \mathscr{R}(T_{d_0 +1})$. By construction, the map $T_{d_0+1}$ is of the desired form.


\subsection{Proof of Theorem \ref{th:intro1}-(a): step 2}\label{S:case1}

We can now assume that the matrix $A$ in the definition \eqref{eq:dot} of $T$ satisfies $d_0 = d-1$, i.e.~$\dim E_k = 1$. Consider $\Psi \in \mathscr{R}(T)$, and assume that it is not cohomologous to a constant. 
Then, by definition of $\mathscr{R}$, we can write $\Psi = \psi_{d-1} \circ \pi_{d-1} + \Psi_d^{\perp}$, where $\psi_{d-1}$ is smoothly cohomologous to a constant w.r.t.~$T_{d-1}$. Thus, there exists a smooth function $u \colon \T^{d-1} \to \R$ such that $\psi_{d-1} - \int \psi_{d-1}= u \circ T_{d-1} - u$. 

We notice that $\Psi_d^{\perp}$ is not a measurable coboundary for $T$. Indeed, if this were not the case and $\Psi_d^{\perp} = v \circ T - v$ for some measurable function $v \colon \T^d \to \R$, we would have
\begin{multline*}
\Psi - \int_{\T^d} \Psi \diff \misura_d = \Psi - \int_{\T^d} \psi_{d-1} \circ \pi_{d-1} \diff \misura_d = \psi_{d-1} \circ \pi_{d-1} - \int_{\T^{d-1}} \psi_{d-1} \diff \misura_{d-1} + \Psi_d^{\perp} \\
=  u \circ T_{d-1} \circ \pi_{d-1} - u \circ \pi_{d-1} + v \circ T - v  = ( u \circ \pi_{d-1} + v) \circ T - (u \circ \pi_{d-1} +v),
\end{multline*}
which is a contradiction since we are assuming that $\Psi$ is not measurably cohomologous to a constant.
The first step is to prove that the Birkhoff sums of $\Psi_d^{\perp}$ grow in measure, namely the following result.

\begin{thm}\label{th:stretchingroof}
For any function $\Psi^{\perp} \in \mathscr{Q}(d)$, which is not a measurable coboundary for $T$, and any $C >1$ we have
$$
\lim_{n \to \infty} \misura_d \left(\modulo{S_n(\Psi^{\perp})}<C\right) = 0.
$$
\end{thm}

From Theorem \ref{th:stretchingroof}, using the fact that $\psi_{d-1}$ is smoothly cohomologous to a constant, we deduce mixing. This final part follows more closely the ideas in \cite{avila:heisenberg}, hence we leave the proof for the appendix.
\begin{thm}\label{th:ulcigrai}
Assume that $d_0 = d-1$. Assume also that $\Psi \in \mathscr{P}(d)$ is not a measurable coboundary for $T$ and that the function $\psi_{d-1}$ defined by \eqref{eq:dec} is smoothly cohomologous to a constant. Then, the suspension flow $\{T_t^{\Psi}\}_{t \in \R}$ is mixing.
\end{thm}


\section{Proof of Theorem \ref{th:intro1}-part b}\label{S:dom}

In this section we prove Theorem \ref{th:intro1}-(b) by constructing the set $\mathscr{M}=\mathscr{M}(d)$, dense in $ \mathscr{P}(d)$ w.r.t.~$\norma{\cdot}_{\infty}$, which consists of roof functions inducing a mixing suspension flow. 
We characterize smooth coboundaries for skew-translations in terms of their Fourier coefficients and we apply Proposition \ref{th:induct} inductively to produce mixing suspension flows in higher dimension, starting from the ones in dimension 2, see \cite[\S5]{avila:heisenberg}.

If we denote again by $\pi_i \colon \T^d \to \T^i$ the projection onto the first $i$ coordinates and by $\pi_i(\mathbf{x}) = \mathbf{x}_i$, we have a sequence of factors
\begin{equation}\label{eq:chain}
(\T^d,T) \mapsto (\T^{d-1}, T_{d-1}) \mapsto \cdots \mapsto (\T^{2}, T_{2}),
\end{equation}
where $T_{i} \mathbf{x}_i = \mathbf{x}_i A_i + \mathbf{b}_i$ and $A_i = (a_{l,m})_{1 \leq l,m \leq i}$.  
\begin{defin}
Let $\Psi \in \mathscr{P}(d)$ be written as
$$
\Psi = \psi_{2} \circ \pi_2 + \Psi^{\perp}_{3} \circ \pi_3+ \cdots + \Psi^{\perp}_{d}, \text{\ with\ } \psi_{2} \in \mathscr{P}(2),\text{\ and\ } \Psi^{\perp}_{i} \in \mathscr{Q}{(i)}, \text{\ for\ } i=3, \dots, d.
$$
We say that $\Psi \in \mathscr{M}(d)$ if $\psi_{2}$ induces a mixing suspension flow for the 2-dimensional skew-translation $(\T^2, T_2)$ and $\Psi^{\perp}_{i}$ is a smooth coboundary for $T_i$ for all $i=3, \dots, d$.
\end{defin}

Every function in $\mathscr{M}(d)$ induces a mixing suspension flow by Remark \ref{rk:mixcob} and Proposition \ref{th:induct} applied inductively in \eqref{eq:chain} up to the last factor. Moreover, the set of mixing roofs $\psi_{2}$ is dense in $\mathscr{P}(2)$ by \cite{avila:heisenberg} and, by Lemma \ref{th:densecob}, the set of smooth coboundaries $\Psi^{\perp}_{i}$ for $T_i$ is dense in $\mathscr{Q}(i)$. Therefore, $\mathscr{M}(d)$ is dense in $\mathscr{P}(d)$, and hence in $\mathscr{C}(\T^{d})$.

We now characterize the set $\mathscr{M}(d)$ so that it is possible to effectively check if a trigonometric polynomial $\Psi$ belongs to $\mathscr{M}(d)$.
The case of $\psi_{2}$ has already been treated in \cite[\S5]{avila:heisenberg}; let us analyze when $\Psi^{\perp}_{i} \in \mathscr{Q}(i)$ is a smooth coboundary for $T_i$.

The following lemma is easy to be verified.
\begin{lemma}
Let $\mathcal{O}_{i}$ be the set of orbits of the action of the transpose $A_i^T$ of $A_i$ on $\Z^i$ and for any $\omega \in \mathcal{O}_{i}$ let 
$$
\mathscr{H}_{\omega} = \bigoplus_{\mathbf{l}\in \omega} \C e({\mathbf{l} \cdot  \mathbf{x}_i }).
$$
The space $L^2(\T^i)$ admits an orthogonal splitting
$$
L^2(\T^i) = \bigoplus_{\omega \in \mathcal{O}_{i}}\mathscr{H}_{\omega},
$$
and all the components are $T_i$-invariant.
\end{lemma}
Therefore, it is enough to investigate the existence of solutions $u$ for the cohomological equation $\Psi^{\perp}_{i} = u \circ T_i - u$ in each component $\mathscr{H}_{\omega}$. The following result is a generalisation in higher dimension of a theorem by Katok \cite[Theorem 11.25]{katok:combinatorial}.

\begin{lemma}\label{th:explicitcob}
Let $\omega \in \mathcal{O}_{i}$; consider $\mathbf{l}^{(0)} \in \omega$ and denote the elements of the orbit $\omega$ by $\mathbf{l}^{(k)} = \mathbf{l}^{(0)} (A_i^T)^k$ for $k \in \Z$.
The function 
$$
\Psi^{\perp}_{i} (\mathbf{x}_i) =  \sum_{\mathbf{l} \in [-m,m]^{i} \cap \omega,\ l_i \neq 0} c_{\mathbf{l}} e(\mathbf{l} \cdot \mathbf{x}_i) \in  \mathscr{Q}(i) \cap \mathscr{H}_{\omega}
$$ 
is a smooth coboundary for $T_i$ if and only if 
\begin{equation}\label{eq:invariantdistrib}
\sum_{k=1}^{N}  c_{\mathbf{l}^{(k)}} e\left( - \sum_{j=0}^{k-1} \mathbf{l}^{(j)} \cdot \mathbf{b}_i \right) +  c_{\mathbf{l}^{(0)}}  +   \sum_{k=1}^{N-1}  c_{\mathbf{l}^{(-k)}} e\left(  \sum_{j=1}^{k} \mathbf{l}^{(-j)} \cdot \mathbf{b}_i \right) = 0,
 \end{equation}
where $N \in \N$ is such that $c_{\mathbf{l}^{(n)}} = 0$ for all $n \geq N$.
\end{lemma}
\begin{proof}
There exists a smooth solution $u$ to the cohomological equation $\Psi^{\perp}_{i} = u \circ T_i - u$ if and only if for every $\mathbf{l} \in [-m,m]^{i} \cap \omega,\ l_i \neq 0$ we have
$$
 \sum_{\mathbf{l} \in [-m,m]^{i} \cap \omega,\ l_i \neq 0} c_{\mathbf{l}} e(\mathbf{l} \cdot \mathbf{x}_i) =  \sum_{\mathbf{l} \in \omega } u_{\mathbf{l}} e(\mathbf{l} \cdot (\mathbf{x}_iA_i+ \mathbf{b}_i)) -  \sum_{\mathbf{l} \in \omega} u_{\mathbf{l}} e(\mathbf{l} \cdot \mathbf{x}_i), 
$$
where $u_{\mathbf{l}}$ are the Fourier coefficients of $u$. Equating coefficients, we get 
$$
c_{\mathbf{l}} = u_{\mathbf{l}(A_i^T)^{-1}} e(\mathbf{l}(A_i^T)^{-1} \cdot\mathbf{b}_i) - u_{\mathbf{l}}
$$
which implies, considering $\mathbf{l}^{(0)} \in \omega$,
\begin{equation*}
\begin{split}
& u_{\mathbf{l}^{(0)}} = u_{\mathbf{l}^{(-1)}} e(\mathbf{l}^{(-1)} \cdot \mathbf{b}_i) - c_{\mathbf{l}^{(0)}} \text{\ \ \ and\ \ \ } u_{\mathbf{l}^{(0)}} = (u_{\mathbf{l}^{(1)}} + c_{\mathbf{l}^{(1)}}) e( - \mathbf{l}^{(0)} \cdot \mathbf{b}_i). 
\end{split}
\end{equation*}
Recursively, for all $N \geq 1$ we obtain
\begin{equation*}
\begin{split}
& u_{\mathbf{l}^{(0)}} = u_{\mathbf{l}^{(-N)}} e\left( \sum_{k=1}^N \mathbf{l}^{(-k)} \cdot \mathbf{b}_i \right) -  \sum_{k=1}^{N-1}  c_{\mathbf{l}^{(-k)}} e\left(  \sum_{j=1}^{k} \mathbf{l}^{(-j)} \cdot \mathbf{b}_i \right) - c_{\mathbf{l}^{(0)}}, \\
& u_{\mathbf{l}^{(0)}} =  u_{\mathbf{l}^{(N)}} e\left( - \sum_{k=0}^{N-1} \mathbf{l}^{(k)} \cdot \mathbf{b}_i \right) + \sum_{k=1}^{N}  c_{\mathbf{l}^{(k)}} e\left( -  \sum_{j=0}^{k-1} \mathbf{l}^{(j)} \cdot \mathbf{b}_i\right). 
\end{split}
\end{equation*}

By assumption, $a_{i,i+1}\neq 0$ and $l_i \neq 0$; hence, for $\modulo{N} \to \infty$, we have $\norma{\mathbf{l}^{(N)} }_{\infty} \geq \modulo{l_{i-1}-a_{i,i+1}Nl_i} \to \infty$. Therefore, if a solution $u$ exists, we have $u_{\mathbf{l}^{(N)}} \to 0$. We obtain two expressions for $u_{\mathbf{l}^{(0)}}$ 
\begin{equation*}
\begin{split}
& u_{\mathbf{l}^{(0)}} = \lim_{N \to \infty} - c_{\mathbf{l}^{(0)}} -   \sum_{k=1}^{N-1}  c_{\mathbf{l}^{(-k)}} e\left(   \sum_{j=1}^{k} \mathbf{l}^{(-j)} \cdot \mathbf{b}_i \right), \text{\ } u_{\mathbf{l}^{(0)}} =  \lim_{N \to \infty} \sum_{k=1}^{N}  c_{\mathbf{l}^{(k)}} e\left( - \sum_{j=0}^{k-1} \mathbf{l}^{(j)} \cdot \mathbf{b}_i \right),
\end{split}
\end{equation*}
which, equated, gives \eqref{eq:invariantdistrib}. We remark that the expressions above are finite sums, since there are only finitely many $k$ such that $c_{\mathbf{l}^{(k)}} \neq 0$.

On the other hand, if \eqref{eq:invariantdistrib} holds, defining $u_{\mathbf{l}}$ as above gives us the Fourier coefficients of the solution $u$ to the cohomological equation.
\end{proof}

\begin{example} 
Consider, for example, a uniquely ergodic skew shift over $\T^3$ of the form
$$
T(x,y,z) = (x,y,z) A + (b_x, b_y, b_z), \text{\ \ \ with\ \ \ } A=
\begin{pmatrix}
1 & 1 & 2 \\
0 & 1 & 2 \\
0 & 0 & 1
\end{pmatrix}
.
$$
First, consider the quotient system $T_2(x,y) = (x,y) A_2 + (b_x, b_y)$, where $A_2 = 
\left(\begin{smallmatrix}
1 & 1\\
0 & 1
 \end{smallmatrix} \right)$. Any function 
$$
\psi_2(x,y)= c_0+ \sum_{0 \leq \modulo{k} \leq m} c_{(k,1)} e(kx+y) 
$$ 
which satisfies
$$
\sum_{0 \leq \modulo{k} \leq m} c_{(k,1)} e \left( \frac{k- k^2}{2}b_x - k b_y \right) \neq 0,
$$
iduces a mixing suspension flow over the quotient system $(\T^2, T_2)$, as shown in \cite[\S2.4]{avila:heisenberg}. 

Straightforward computations give us
$$
\sum_{j=0}^{k-1} (A^T)^j 
= \sum_{j=0}^{k-1} 
\begin{pmatrix}
1 & 0 & 0 \\
j & 1 & 0 \\
j^2+j & 2j & 1
\end{pmatrix}
=
\begin{pmatrix}
k & 0 & 0 \\
-\frac{k^2-k}{2} & k & 0 \\
\frac{k^3-k}{3} &  k^2-k  & k
\end{pmatrix}
$$
and 
$$
\sum_{j=1}^{k} ( A^T)^{-j} 
= \sum_{j=1}^k 
\begin{pmatrix}
1 & 0 & 0 \\
-j & 1 & 0 \\
j^2-j & -2j & 1
\end{pmatrix}
=
\begin{pmatrix}
k & 0 & 0 \\
-\frac{k^2+k}{2} & k & 0 \\
-\frac{k^3-k}{3} & -k^2-k & k
\end{pmatrix},
$$
for all $k \geq 1$. Fix $\mathbf{l}^{(0)} = (0,0,1)$, then $\mathbf{l}^{(k)} = (k^2+k,2k,1 )$. By Lemma \ref{th:explicitcob}, any function
$$
\Psi_3^{\perp}(x,y,z) =  \sum_{0 \leq \modulo{k} \leq m} c_{(k^2+k,2k,1 )} e((k^2+k)x+2ky+z)
$$
satisfying
\begin{equation}\label{eq:claim}
 \sum_{0 \leq \modulo{k} \leq m}  c_{(k^2+k,2k,1 )} e\left( - \left(\frac{k^3-k}{3}b_x + (k^2-k)b_y +k b_z \right) \right) =0
 \end{equation}
is a smooth coboundary for $T$. Proposition \ref{th:induct} implies that $\Psi = \psi_2 + \Psi_3^{\perp} \in \mathscr{M}(3)$ induces a mixing suspension flow over $(\T^3,T)$.
\end{example}


\section{Proof of Proposition \ref{th:induct}}\label{proofofwrap}

We show that, under the assumption of Proposition \ref{th:induct}, if the quotient suspension flow $\{ \widehat{T}_t^{\psi} \}_{t \in \R}$ is mixing, then $\{ T_t^{\psi \circ \pi} \}_{t \in \R}$ is mixing.


\subsection{Preliminaries}
Let us denote $\Psi := \psi \circ \pi$, which we remark is constant along the $x_d$-coordinate, and assume $\int \Psi = 1$. Let $c$ and $C$  be its minimum and maximum respectively. 
Consider $Q = \prod_{j=1}^{d} [w_j,w_j'] \times [q_1, q_2]$ and $R = \prod_{j=1}^{d} [v_j,v_j'] \times [r_1, r_2]$ two cubes in $\{ (\mathbf{x}, r) : \mathbf{x} \in \T^d \text{ and } 0 \leq r <\Psi(\mathbf{x})\}$; it is sufficient to prove mixing for sets of this form. 
Denote by $\widehat{Q}= \pi(Q), \widehat{R}=\pi(R)$ the corresponding cubes in the quotient system, namely $\widehat{Q} = \prod_{j =1}^{d-1} [w_j,w_j'] \times [q_1, q_2]$ and $\widehat{R} =  \prod_{j =1}^{d-1}  [v_j,v_j'] \times [r_1, r_2]$. For any $\varepsilon >0$, define
\begin{equation*}
\begin{split}
&\widehat{Q}_{-\varepsilon} = \prod_{j=1}^{d-1}[w_j+\varepsilon,w_j' - \varepsilon] \times [q_1 + \varepsilon, q_2-\varepsilon]  \subset \widehat{Q}, \\ 
&\widehat{R}_{-\varepsilon} = \prod_{j=1}^{d-1}[v_j+\varepsilon,v_j' - \varepsilon] \times [r_1 + \varepsilon, r_2-\varepsilon]  \subset \widehat{R}.
\end{split}
\end{equation*}
Let $\mathbf{v} \in \Z^d$ be such that $\mathbf{v}A = \mathbf{v} + a \mathbf{e}_d$, with $a \neq 0$. Up to changing $\mathbf{v}$ with $\mathbf{v} - (\mathbf{v} \cdot \mathbf{e}_d) \mathbf{e}_d$ and up to rescaling, we can assume that $\mathbf{v} \cdot \mathbf{e}_d = 0$ and the coordinates of $\mathbf{v}$ are coprime. Denote by $\partial_{\mathbf{v}}$ the directional derivative along $\mathbf{v}$, namely, if $f \in \mathscr{C}^1(\T^d)$, let $\partial_{\mathbf{v}} f = \nabla f \cdot \mathbf{v}$, where $\nabla f$ is the gradient of $f$.
Fix $\varepsilon >0$ and choose $0<\varepsilon_0<1$ such that $3(d C + 1) \varepsilon_0 < \varepsilon$. 
%
Recalling \eqref{eq:birksum}, let $S_n({\partial}_{\mathbf{v}}\Psi)$ be the Birkhoff sum up to $n$ of the derivative of $\Psi$ along ${\mathbf{v}}$.
By Birkhoff Ergodic Theorem, since $T$ is uniquely ergodic and ${\partial}_{\mathbf{v}}\Psi$ has zero average, there exists $N \geq 1$ such that for all $n \geq N$ we have 
\begin{equation}\label{eq:bet}
\frac{1}{n} S_n({\partial}_{\mathbf{v}}\Psi)(\mathbf{x}) \leq \frac{ac}{2C} \varepsilon_0,
\end{equation}
for all $\mathbf{x} \in \T^{d}$.

For every $\mathbf{x}, \mathbf{x}\rq{} \in \T^d$, from the definition \eqref{eq:sflow} of $n_t(\mathbf{x})$ it follows immediately that 
$$
n_t(\mathbf{x}) c \leq S_{n_t(\mathbf{x})}(\Psi)(\mathbf{x}) \leq t < S_{n_t(\mathbf{x}\rq{})+1}(\Psi)(\mathbf{x}\rq{}) \leq (n_t(\mathbf{x}\rq{})+1) C \leq 2 n_t(\mathbf{x}\rq{}) C,
$$
for all $t > C$. Then, we have $n_t(\mathbf{x}) /  n_t(\mathbf{x}\rq{})  \leq 2 C/ c$.
Choose $\overline{t} >0$ such that for all $t \geq \overline{t}$
\begin{equation}\label{eq:conditions}
\begin{split}
\text{(i)\ } & t \geq (N+1) C \text{\ \ \ so that\ \ \ } n_t(\mathbf{x}) > t/ C - 1 \geq  N; \\
\text{(ii)\ } &  \frac{\norma{\mathbf{v}} C}{a(t-C)} \leq \varepsilon_0 \text{\ \ \ so that\ \ \ } \frac{\norma{\mathbf{v}} }{ n_t(\mathbf{x}) a } < \frac{ \norma{\mathbf{v}} C}{a(t-C)} \leq \varepsilon_0; \\
\text{(iii)\ } &  \modulo{\misura \left( \widehat{T}_t^{\psi} \big(\widehat{R}_{-\varepsilon_0} \big)\cap \widehat{Q}_{-\varepsilon_0} \right) - \misura \big( \widehat{R}_{-\varepsilon_0} \big) \misura \big( \widehat{Q}_{-\varepsilon_0}\big)} \leq \varepsilon_0.
\end{split}
\end{equation}
The third condition above is guaranteed by mixing of the suspension flow $\{ \widehat{T}_t^{\psi} \}_{t \in \R}$ on the quotient $\T^{d-1}$.


\subsection{Wrapping segments}
We now consider segments of length less than $\varepsilon_0$ parallel to $\mathbf{v}$ contained in $R$ and we study their evolution after sufficiently large time $t$. 
Recalling \eqref{eq:factor}, fix $t \geq \overline{t}$ and consider a point $\mathbf{r} = (\mathbf{x},r) \in R$ such that $\mathbf{r} + ( n_t(\mathbf{r}) a )^{-1} \mathbf{v} \in R$. 
Let $\gamma_{\mathbf{r}}(s) = \mathbf{r} + s \mathbf{v} $, with $0 \leq s \leq \overline{s}= (n_t(\mathbf{r}) a)^{-1}$, be the segment parallel to $\mathbf{v}$ starting from $\mathbf{r}$ of length $ \overline{s}$. 
Condition \eqref{eq:conditions}-(ii) ensures that the length of  $\gamma_{\mathbf{r}}$ is less than $\varepsilon_0$ so that, by hypothesis, it is all contained in $R$, see Figure \ref{fig:3}. We will prove that, if there exists a point of $\pi \circ T_t^{\Psi} (\gamma_{\mathbf{r}})$ which is contained in $ \widehat{Q}_{-\varepsilon_0}$, then all the curve is contained in $\widehat{Q}$.

Let us denote  by $\Gamma_{\mathbf{r}}(s) = T_t^{\Psi} (\gamma_{\mathbf{r}}(s))$ the image of $\gamma_{\mathbf{r}}(s)$ under $T_t^{\Psi}$ and let us compute its tangent vector $\partial_s  \Gamma_{\mathbf{r}}(s) $ at a generic point. For almost every $s$, the value $n_t(\gamma_{\mathbf{r}}(s))$ is locally constant; from the definition \eqref{eq:sflow}, we get
\begin{equation}\label{eq:tgvec}
\begin{split}
\partial_s  \Gamma_{\mathbf{r}}(s)  &= \partial_s \Big(T^{n_t(\gamma_{\mathbf{r}}(s))}\gamma_{\mathbf{r}}(s),\ t - S_{n_t(\gamma_{\mathbf{r}}(s))}(\Psi)(\gamma_{\mathbf{r}}(s)) \Big) \\
&= \Bigg(\mathbf{v} A^{n_t(\gamma_{\mathbf{r}}(s))},\ - \sum_{j=0}^{n_t(\gamma_{\mathbf{r}}(s))-1} \nabla \Psi \circ T^j (\gamma_{\mathbf{r}}(s)) \cdot \mathbf{v} A^j \Bigg) \\
&= \Bigg(  \mathbf{v} + n_t(\gamma_{\mathbf{r}}(s)) a \mathbf{e}_d,\  - \sum_{j=0}^{n_t(\gamma_{\mathbf{r}}(s))-1} (\partial_{\mathbf{v}}\Psi) \circ T^j (\gamma_{\mathbf{r}}(s))  \Bigg),
\end{split}
\end{equation}
where we used the fact that the partial derivative of $\Psi= \psi \circ \pi$ in the $d$-th variable is zero, since $\Psi$ is constant along the $x_d$-coordinate.

We first show that the function $s \mapsto n_t(\gamma_{\mathbf{r}}(s)) $ is constant. In order to do this, we estimate the maximal distance in the $t$-coordinate between two points in the curve $\Gamma_{\mathbf{r}}(s) $.
By definition and \eqref{eq:tgvec}, it equals
$$
\max_{0 \leq s\rq{},s\rq{}\rq{} \leq \overline{s}} \modulo{(\Gamma_{\mathbf{r}}(s') - \Gamma_{\mathbf{r}}(s'') ) \cdot \mathbf{e}_{d+1}} \leq \int_0^{\overline{s}} \modulo{ \sum_{j=0}^{n_t(\gamma_{\mathbf{r}}(s)) -1} (\partial_{\mathbf{v}}\Psi) \circ T^j (\gamma_{\mathbf{r}}(s)) } \diff s. 
$$
From the choice of $N$, \eqref{eq:bet} and \eqref{eq:conditions}-(i), it follows
\begin{equation*}
\begin{split}
& \max_{0 \leq s\rq{},s\rq{}\rq{} \leq \overline{s}} \modulo{( \Gamma_{\mathbf{r}}(s') - \Gamma_{\mathbf{r}}(s'') ) \cdot \mathbf{e}_{d+1}} \leq \int_0^{\overline{s}} \modulo{ S_{n_t(\gamma_{\mathbf{r}}(s))}(\partial_{\mathbf{v}}\Psi)(\gamma_{\mathbf{r}}(s)) } \diff s \\
& \qquad \leq \frac{ac}{2C}\varepsilon_0 \int_0^{\overline{s}} n_t(\gamma_{\mathbf{r}}(s))\diff s \leq \frac{c}{2C} \varepsilon_0 \frac{\max_s n_t(\gamma_{\mathbf{r}}(s)) }{  n_t(\mathbf{r})}  \leq \varepsilon_0.
\end{split}
\end{equation*}
In a similar way, using \eqref{eq:conditions}-(ii), the maximal distance in any other coordinate $x_i$ for $1 \leq i \leq d-1$ between two points in $\Gamma_{\mathbf{r}}(s)$ can be bounded by 
$$
\max_{0 \leq s\rq{},s\rq{}\rq{} \leq \overline{s}} \modulo{(\Gamma_{\mathbf{r}}(s') - \Gamma_{\mathbf{r}}(s'') ) \cdot \mathbf{e}_i} \leq \norma{\mathbf{v}} \int_0^{\overline{s}} \diff s = \frac{\norma{\mathbf{v}}}{n_t(\mathbf{r}) a } \leq \varepsilon_0.
$$ 
In particular, if $\pi ( \gamma_{\mathbf{r}}(s)) \in \widehat{T}_{-t}^{\psi}(\widehat{Q}_{-\varepsilon_0})$ for some $0 \leq s \leq \overline{s}$, then $\pi  \circ \Gamma_{\mathbf{r}}(s) \subset \widehat{Q}$ and therefore we deduce that $n_t$ is constant along $\gamma_{\mathbf{r}}(s)$ and equal to $n_t(\mathbf{r})$, see Figure \ref{fig:3}.

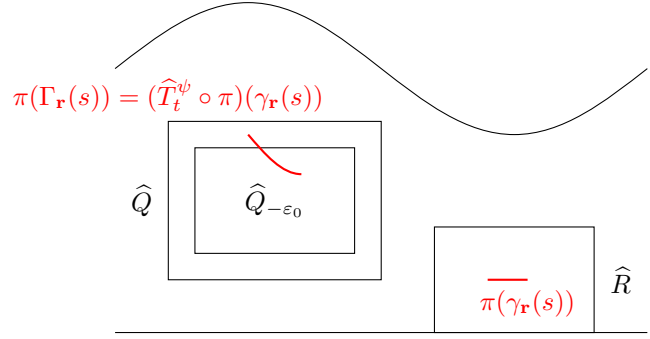
\begin{figure}[h!]
\centering
\begin{tikzpicture}[scale=3.5]
\clip (-0.8,-0.4) rectangle (3,1.5);
\draw (0,0) -- (2,0);
\draw (0,1) sin (0.5,1.25) cos (1,1) sin (1.5,0.75) cos (2,1);
\draw (0.2, 0.2) rectangle (1,0.8);
\draw[white] (0.1,0.5)  node{\textcolor{black}{$\widehat{Q}$}};
\draw (0.3, 0.3) rectangle (0.9,0.7);
\draw[white] (0.6,0.5)  node{\textcolor{black}{$\widehat{Q}_{-\varepsilon_0}$}};
\draw (1.2,0) rectangle (1.8,0.4);
\draw[white] (1.9,0.2)  node{\textcolor{black}{$\widehat{R}$}};
\draw[thick, red] (1.4,0.2) -- (1.55,0.2) node [anchor=north] {\textcolor{red}{$\pi (\gamma_{\mathbf{r}}(s))$}};
\draw[thick, red] (0.5,0.75) sin (0.7,0.6);
\draw[white] (0.2, 0.9)  node{\textcolor{red}{$ \pi (\Gamma_{\mathbf{r}}(s)) = ( \widehat{T}_t^{\psi} \circ \pi) (\gamma_{\mathbf{r}}(s))$}};
\end{tikzpicture}
\caption{Quotient system $(\T^{d-1}, \widehat{T})$: if some point of the curve $ \pi (\Gamma_{\mathbf{r}}(s))$ is contained in $\widehat{Q}_{-\varepsilon_0}$, then the whole curve is contained in $\widehat{Q}$.}
\label{fig:3}
\end{figure}

Since $n_t(\gamma_{\mathbf{r}}(s)) =n_t(\mathbf{r})$, by \eqref{eq:tgvec} the speed in the $x_d$-coordinate is constant and equal to $n_t(\mathbf{r})a$. Moreover, the distance in the $x_{d}$-coordinate of the endpoints of $\Gamma_{\mathbf{r}}(s)$ is equal to 
$$
\modulo{\int_0^{\overline{s}}   n_t(\gamma_{\mathbf{r}}(s)) a  \diff s} = \modulo{\int_0^{\overline{s}}  n_t(\mathbf{r}) a \diff s} = \modulo{\overline{s}  n_t(\mathbf{r}) a } = 1. 
$$


\subsection{Final estimates}

In order to estimate the measure of $R \cap T_{-t}^{\Psi} (Q)$, we want to apply Fubini\rq{}s Theorem and integrate along each circle parallel to $\mathbf{v}$. Indeed, the torus $\T^d$ is a circle bundle over a closed submanifold $W$ isomorphic to a $(d-1)$-dimensional torus with fibers parallel to $\mathbf{v} \in \Z^d$. Let us consider the corresponding decomposition of the Lebesgue measure as product measure, namely $\diff \misura_d = \diff \mathbf{v}\wedge \omega$, where $\omega$ is a volume form over $W$.

Let 
$$
S = \left\{ \mathbf{r} \in R : \pi(\mathbf{r}) \in \widehat{R}_{-\varepsilon_0} \cap  \widehat{T}_{-t}^{\psi}(\widehat{Q}_{-\varepsilon_0})\right\}.
$$
We want to consider all segments $\gamma_{\mathbf{r}}(s)$ that contain at least one point in $S$. For any point $\mathbf{w} \in W$, let us partition the fiber $\mathbf{w} + [0,1) \mathbf{v}$ into segments of length $\overline{s}$; more precisely define
$$
\mathbf{r}_0(\mathbf{w}) = \mathbf{w},\ \mathbf{r}_1(\mathbf{w})= \mathbf{r}_0 + (n_t(\mathbf{r}_0)a)^{-1} \mathbf{v}, \dots, \mathbf{r}_{i+1}(\mathbf{w})= \mathbf{r}_i + (n_t(\mathbf{r}_i)a)^{-1} \mathbf{v}, 
$$
up to the largest $i$ such that $\sum_i (n_t(\mathbf{r}_i)a)^{-1} <1$, and let $R_{-}(t)$ be the union for $\mathbf{w} \in W$ of all segments $\gamma_{\mathbf{r}_i(\mathbf{w})}(s)$ which contain at least one point in $S$.
Notice that $ \misura_d(S) \leq \misura_d(R_{-}(t)) $; moreover, recalling the definition of $R$, by Fubini\rq{}s Theorem,
\begin{equation}\label{eq:RandS}
\begin{split}
& \misura_d(S) = \misura_{d-1} \left( \widehat{R}_{-\varepsilon_0} \cap \widehat{T}_{-t}^{\psi}(\widehat{Q}_{-\varepsilon_0}) \right) \modulo{v_d\rq{} - v_d}, \\ 
& \misura_d(R_{-}(t)) = \int_{\T^d} \one_{R_{-}(t)} \diff \mathbf{v}\wedge \omega = \int_W \Bigg( \sum_{i : \gamma_{\mathbf{r}_i(\mathbf{w})} \subset R_{-}(t)}  (n_t(\mathbf{r}_i)a)^{-1} \Bigg) \omega(\mathbf{w}).
\end{split}
\end{equation}


By definition of $R_{-}(t)$ and $S$, we have $R_{-}(t) \subset R$, thus
\begin{equation}\label{eq:Rminus}
\begin{split}
\misura_d( T_t^{\Psi}(R) \cap Q) &= \int_{\T^d} (\one_{R} \circ T_{-t}^{\Psi}) \cdot \one_Q \diff \misura_d \geq \int_{\T^d} \one_{R_{-}(t)} \cdot (\one_Q \circ T_t^{\Psi} )\diff \misura_d\\
& = \int_W \Bigg( \sum_{i : \gamma_{\mathbf{r}_i(\mathbf{w})} \subset R_{-}(t)} \int_0^{\overline{s}}\one_Q \circ T_t^{\Psi} \circ \gamma_{\mathbf{r}_i(\mathbf{w})} (s)\diff s \Bigg) \omega(\mathbf{w}).
\end{split}
\end{equation}
For each curve $\gamma_{\mathbf{r}_i(\mathbf{w})} \subset R_{-}(t)$, by definition of $R_{-}(t)$, there exists a point $\gamma_{\mathbf{r}_i(\mathbf{w})}(s)$ contained in $S$, so that $\pi(\Gamma_{\mathbf{r}_i(\mathbf{w})}(s)) \in \widehat{Q}$. Hence, the point $\Gamma_{\mathbf{r}_i(\mathbf{w})}(s) \in Q$ if and only if its $x_d$-coordinate $\Gamma_{\mathbf{r}_i(\mathbf{w})}(s) \cdot \mathbf{e}_d$ is in $[w_d,w_d\rq{}]$. Since the speed of $\Gamma_{\mathbf{r}_i(\mathbf{w})}$ in this latter direction is constant and equal to $\overline{s}^{-1} = n_t(\mathbf{r}_i(\mathbf{w}))a$, we get
\begin{equation}\label{eq:Q}
\int_0^{\overline{s}}\one_Q \circ T_t^{\Psi} \circ \gamma_{\mathbf{r}_i(\mathbf{w})}(s)\diff s = \overline{s} \modulo{w_d\rq{} - w_d}.
\end{equation}
Combining \eqref{eq:Q} with \eqref{eq:Rminus} and \eqref{eq:RandS}, we obtain
\begin{equation*}
\begin{split}
\misura_d( T_t^{\Psi}(R) \cap Q) & \geq \misura_d(R_{-}(t)) \modulo{w_d\rq{} - w_d} \geq \misura_d(S) \modulo{w_d\rq{} - w_d} \\
&= \misura_{d-1} \left( \widehat{R}_{-\varepsilon_0} \cap \widehat{T}_{-t}^{\psi}(\widehat{Q}_{-\varepsilon_0}) \right) \modulo{v_d\rq{} - v_d} \modulo{w_d\rq{} - w_d}.
\end{split}
\end{equation*}

The area of a face of $Q$ is less than $C = \max \Psi > 1$, thus we can bound $\misura_{d-1} (\widehat{Q}_{-\varepsilon_0}) \geq \misura_{d-1} (\widehat{Q}) - (3d) C \varepsilon_0$, where $3d$ is the number of faces of $Q$. 
Using \eqref{eq:conditions}-(iii), we get
\begin{equation*}
\begin{split}
 \misura_d & (T_t^{\Psi}(R) \cap Q) \geq ( \misura_{d-1}(\widehat{R}_{-\varepsilon_0}) \misura_{d-1} (\widehat{Q}_{-\varepsilon_0}) - \varepsilon_0) \modulo{v_d\rq{} - v_d} \modulo{w_d'-w_d}\\
 & \geq \big( ( \misura_{d-1}(\widehat{R})- 3d C \varepsilon_0 )(\misura_{d-1}(\widehat{Q}) - 3d C \varepsilon_0)-\varepsilon_0 \big) \modulo{v_d'-v_d}  \modulo{w_d'-w_d} \\
& \geq \misura_d(R) \misura_d(Q) -3 d C (\misura_d(R)+\misura_d(Q))\varepsilon_0 - \varepsilon_0 \geq \misura_d(R) \misura_d(Q) -\varepsilon,
\end{split}
\end{equation*}
by the choice of $\varepsilon$. 
The other inequality can be derived in a similar way:
one considers $R_{+}(t)$ instead of $R_{-}(t)$, where $R_{+}(t)$ is defined analogously to $R_{-}(t)$ as the union of the segments $\gamma_{\mathbf{r}}(s)$ which contain at least one point that belongs to $S\rq{} = R \cap\pi^{-1}( \widehat{T}_{-t}^{\psi}  (\widehat{Q}) )$; then, one notices that
$$
R_{+}(t) \subset R_{+ \varepsilon} \cap \pi^{-1}\left( \widehat{T}_{-t}^{\psi}  (\widehat{Q}_{+\varepsilon} ) \right), 
$$
where 
$$
\widehat{Q}_{+\varepsilon} = \prod_{j=1}^{d-1}[w_j-\varepsilon,w_j' + \varepsilon] \times [q_1 - \varepsilon, q_2 + \varepsilon] \supset \widehat{Q},
$$
and similarly for $R_{+\varepsilon} $. Finally, it is sufficient to estimate 
$
\misura_d\big(  R\cap  T_{-t}^{\Psi}(Q) \big) = \misura_d \big( S\rq{} \cap  T_{-t}^{\Psi}(Q) \big) \leq \misura_d \big( R_{+}(t) \cap  T_{-t}^{\Psi}( Q_{+\varepsilon})  \big)
$
by applying Fubini\rq{}s Theorem as above.
The proof is therefore complete.


\section{Proof of Theorem \ref{th:stretchingroof}}\label{ST1}

We now suppose that $E_k = \text{Im} (A-\Id)^k = \langle \mathbf{e}_d \rangle$ and $\text{Im} (A - \Id)^{k+1} = \{0 \}$.
Let $\Psi^{\perp} \in \mathscr{Q}(d)$ and, denoting $e(x) = \exp(2 \pi i x)$, write
\begin{equation}\label{eq:Psifinal}
\Psi^{\perp}(\mathbf{x}) = \sum_{\mathbf{l} \in [-m,m]^d\cap \Z^d} c_{\mathbf{l}}e(\mathbf{l} \cdot \mathbf{x}).
\end{equation}
Let us assume that $\Psi^{\perp} \in \mathscr{Q}(d)$ is not a measurable coboundary for $T$; we prove that Birkhoff sums $S_n(\Psi^{\perp})$ of $\Psi^{\perp}$ grow in measure. In order to do this, we first apply a classical Gottschalk-Hedlund argument to prove that they grow in average (Lemma \ref{th:cesaro}) and then a decoupling result (Lemma \ref{th:decoupling}), which generalizes \cite[Lemma 5]{avila:heisenberg} to higher dimension. The key observation is that, due to the form of the skew-translation $T$, for large $N \geq 1$ the divergence of nearby points happens mostly in the $x_d$-direction, namely it is of higher order than in the other coordinates.

Denote by $\widehat{\mathbf{x}} := \pi(\mathbf{x}) \in \T^{d-1}$ the projection of $\mathbf{x} \in \T^d$ onto the first $d-1$ coordinates; the projection $\pi$ gives a factor $(\T^{d-1},\widehat{T})$ of $(\T^d, T)$.

\begin{remark}\label{rk:one}
For any $N \geq 1$, we can express the $N$-th iterate of $T$ as $T^{N}\mathbf{x} = \mathbf{x} A^{N} + \mathbf{b}({N})$, where $\mathbf{b}({N}) =(b_1({N}), \dots, b_d({N}))= \sum_{i=0}^{{N}-1} \mathbf{b}A^i$ and $A^{N}=(a_{i,j}({N}))_{i,j}$ is an upper triangular unipotent matrix. For any ${N} \geq k+1$, we can write $A^{N} = (\Id + (A-\Id))^{N} = \sum_{i=0}^{k} {{N}\choose{i}} (A-\Id)^i$. It follows that each nonzero entry $a_{i,j}({N})$ is a polynomial in ${N}$ of degree $\leq k$. Moreover, since $E_k = \langle \mathbf{e}_d \rangle$, the only terms $a_{i,j}({N})$ of order $O(N^k)$ are in the last column, namely for $j=d$.
With this notation, we have
$$
T^{N}\mathbf{x} = T^{N} (\widehat{\mathbf{x}},x_d) = \big(\widehat{T}^{N}\widehat{\mathbf{x}},\ x_d + x_{d-1} a_{d-1,d}(N) + \cdots + x_1 a_{1,d}(N) + b_d({N}) \big).
$$ 
\end{remark}

\begin{lemma}[{\cite[Corollary 1]{avila:heisenberg}}]\label{th:cesaro}
For any $C>1$ we have
\begin{equation}\label{eq:Cesaroav}
\lim_{N \to \infty} \frac{1}{N} \sum_{n=0}^{N-1} \misura_d \left(\modulo{S_n(\Psi^{\perp})}<C\right) = 0.
\end{equation}
In particular, for any $\varepsilon >0$ there exist arbitrarily long arithmetic progressions $\{i\overline{n}\}_{i=1}^{\ell}$ such that $\misura_d(\modulo{S_{i\overline{n}}(\Psi^{\perp})}<C)<\varepsilon$.
\end{lemma}
\begin{proof}
The proof of the first statement is the same as in \cite[Corollary 1]{avila:heisenberg}. For the reader\rq{}s convenience, we present a proof of the second part.
Fix $\varepsilon >0$ and let 
$$
B_{\varepsilon} = \left\{ n \in \N:   \misura_d \left(\modulo{S_n(\Psi^{\perp})}<C\right) \geq \varepsilon \right\} \subset \N.
$$
By \eqref{eq:Cesaroav}, $B_{\varepsilon}$ has zero density, see, e.g., \cite[Theorem 2.8.1]{choudary:real}. 

Let us consider $\ell \geq 1$, $0< \delta < 2/(\ell^2 + \ell)$ and $N_0 \geq 1$ such that for all $N \geq N_0$ we have $\# \{n \in B_{\varepsilon}: n \leq N \} \leq \delta N$. Fix $N \geq N_0$; we want to find $\overline{n} \leq N$ such that $\overline{n}, 2\overline{n}, \dots, \ell \overline{n} \in \N \setminus B_{\varepsilon}$. Equivalently, if we denote by $ B_{\varepsilon} / j := \{ b/j : b \in B_{\varepsilon} \} \subset \mathbb{Q} $, we look for $1 \leq \overline{n} \leq N$ such that 
$$
\overline{n} \notin \{ 1, \dots, N \} \cap  \frac{B_{\varepsilon}}{j} \text{\ \ \ for all\ } j = 1, \dots, \ell.
$$
We estimate the cardinality
\begin{equation*}
\begin{split}
& \# \Bigg(\{ 1,\dots, N \} \setminus \bigcup_{j=1}^{\ell} \{ 1, \dots, N \} \cap  \frac{B_{\varepsilon}}{j} \Bigg) \geq N - \sum_{j=1}^{\ell} \# \Bigg( \{ 1, \dots, N \} \cap  \frac{B_{\varepsilon}}{j} \Bigg)\\
& \qquad \geq N - \sum_{j=1}^{\ell} \# \Big( \{ 1, \dots, jN \} \cap B_{\varepsilon} \Big) \geq N \left( 1- \frac{\ell (\ell+1)}{2} \delta \right) >0,
\end{split}
\end{equation*}
by the choice of $\delta$. In particular, the set $\{ 1 \leq \overline{n} \leq N : j \overline{n} \notin B_{\varepsilon},\text{\ for\ } j=1,\dots, \ell \}$ is not empty and the claim follows.
\end{proof}


\subsection{Decoupling}
The following is our decoupling result.

\begin{lemma}\label{th:decoupling}
Let $C>1$ and $\varepsilon >0$. There exist $C\rq{}>1 $ and $\varepsilon\rq{}>0$ such that for all $n \geq 1$ satisfying $\misura_d(\modulo{S_n(\Psi^{\perp})}<C\rq{}) < \varepsilon\rq{}$ there exists $N_0 \geq 1$ such that for all $N \geq N_0$ we have
\begin{equation}\label{eq:concl}
\misura_d \left( \modulo{S_N(\Psi^{\perp}) \circ T^n - S_N(\Psi^{\perp})}<2C \right) < \varepsilon.
\end{equation} 
\end{lemma}
\begin{proof}
First of all, by the cocycle relation for Birkhoff sums, we notice that $S_N(\Psi^{\perp}) \circ T^n- S_N(\Psi^{\perp}) = S_{N+n}(\Psi^{\perp}) - S_n(\Psi^{\perp}) - S_N(\Psi^{\perp}) = S_n(\Psi^{\perp}) \circ T^N- S_n(\Psi^{\perp})$. We want to compare $\modulo{S_n(\Psi^{\perp}) \circ T^N - S_n(\Psi^{\perp}) }$ with $\modulo{S_n(\Psi^{\perp})}$, which, by hypothesis, is larger than $C\rq{}$ up to a set of measure at most $\varepsilon\rq{}$, the latter constants still to be determined.

We denote by $\mathbf{a}_j({N})$ the transpose of the $j$-th column of $A^N$ and the translation vector by $\mathbf{b}(N) = (b_1(N), \dots, b_{d}(N))$. Let $\widehat{\mathbf{a}}_{d}(N) = \pi (\mathbf{a}_d(N)) = (a_{1,d}(N), \dots, a_{d-1,d}(N))$ be the vector obtained from $\mathbf{a}_{d}(N)$ by suppressing the last coordinate $a_{d,d}(N)=1$. 

From \eqref{eq:Psifinal}, write
$$
\Psi^{\perp}(\mathbf{x}) = \sum_{0 < \modulo{l} \leq m} c_{l}(\widehat{\mathbf{x}})e(l x_d).
$$
Using Remark \ref{rk:one}, we can express the Birkhoff sum of $\Psi^{\perp}$ as
\begin{equation}\label{eq:psienne}
S_n(\Psi^{\perp})(\mathbf{x}) = \sum_{r=0}^{n-1}\ \sum_{0 < \modulo{l} \leq m} c_l(\widehat{T}^r\widehat{\mathbf{x}})   e \big(l (x_d+\widehat{\mathbf{x}} \cdot \widehat{\mathbf{a}}_{d}(r) + b_{d}(r)) \big) =  \sum_{0 < \modulo{l} \leq m} c_{l,n}(\widehat{\mathbf{x}}) e(lx_d),
\end{equation}
where we have denoted
$$
c_{l,n}(\widehat{\mathbf{x}}) = \sum_{r=0}^{n-1} c_l(\widehat{T}^r \widehat{\mathbf{x}}) e \big(l (\widehat{\mathbf{x}} \cdot \widehat{\mathbf{a}}_{d}(r) + b_{d}(r)) \big).
$$
Therefore, we can write
$$
(S_n(\Psi^{\perp}) \circ T^N- S_n(\Psi^{\perp}))(\mathbf{x}) = \sum_{0 < \modulo{l} \leq m} c_{l,n,N}(\widehat{\mathbf{x}}) e(lx_d),
$$
where
\begin{equation}\label{eq:clnn}
c_{l,n,N}(\widehat{\mathbf{x}}) = c_{l,n}(\widehat{T}^N\widehat{\mathbf{x}}) e \big(l(\widehat{\mathbf{x}} \cdot \widehat{\mathbf{a}}_{d}(N) + b_{d}(N)) \big) - c_{l,n}(\widehat{\mathbf{x}}).
\end{equation}
We will now estimate the measure of the set where the modulus of the coefficients $\clznn$ is comparable to $\clzn$. The idea is the following: we first partition $\T^{d-1}$ into sets on which the coefficients $\clzn$ and $\clzn \circ \widehat{T}^N$ are almost constant.
We then show that on a large set there are no cancellations for $\clznn$ by using the fact that the factor $e(l(\widehat{\mathbf{x}} \cdot \widehat{\mathbf{a}}_{d}(N)))$ is of higher order, namely $O(N^{k})$.

\medskip

Let $n \geq 1$ be fixed. The functions $c_{l,n}$ are uniformly continuous, hence let $\delta >0$ be such that if $\norma{\widehat{\mathbf{x}} - \widehat{\mathbf{x}}'} \leq \delta$ then $\modulo{ c_{l,n}(\widehat{\mathbf{x}}) - c_{l,n}(\widehat{\mathbf{x}}')} \leq 1/4$. By Remark \ref{rk:one}, 
$$
\norma{\widehat{T}^N \widehat{\mathbf{x}} - \widehat{T}^N \widehat{\mathbf{x}}'}_{\infty} \leq \norma{\widehat{\mathbf{x}} - \widehat{\mathbf{x}}'}_{\infty} \norma{\widehat{A}^N}_{\infty} = \norma{\widehat{\mathbf{x}} - \widehat{\mathbf{x}}'}_{\infty} O(N^{k-1}).
$$
Let $N_0 \geq 1$ be such that for all $N \geq N_0$, if  $\norma{\widehat{\mathbf{x}} - \widehat{\mathbf{x}}'}_{\infty} \leq (N^{k-1}\log N)^{-1}$ then the term above is less than $\delta$, so that 
\begin{equation}\label{eq:cubes}
\modulo{ \clzn(\widehat{T}^N\widehat{\mathbf{x}}) - \clzn(\widehat{T}^N\widehat{\mathbf{x}}')} \leq 1/4.
\end{equation}

Partition $\T^{d-1}$ into cubes with edges of length $L=( N^{k-1} \log N \sqrt{d-1})^{-1}$ and one face $F$ orthogonal to $\widehat{\mathbf{a}}_{d}(N)$. If $\widehat{\mathbf{x}}$ and $\widehat{\mathbf{x}}\rq{}$ are in one of such cubes, which we will denote by $Q$, then $\norma{\widehat{\mathbf{x}} - \widehat{\mathbf{x}}'}_{\infty} \leq \sqrt{d-1}L$ and so \eqref{eq:cubes} holds.
Fix $Q$ and let $\overline{\mathbf{x}}$ be one of its vertices. Let
$
c_1 = \clzn (\overline{\mathbf{x}}) $ and $
c_2 = \clzn(\widehat{T}^N \overline{\mathbf{x}}) e(l b_{d}(N))$; then for any $\widehat{\mathbf{x}} \in Q$, by \eqref{eq:clnn} and \eqref{eq:cubes},
\begin{equation*}
\begin{split}
\modulo{\clznn(\widehat{\mathbf{x}})} \geq & \modulo{\clzn(\widehat{T}^N \overline{\mathbf{x}}) e \big(l(\widehat{\mathbf{x}} \cdot \widehat{\mathbf{a}}_{d}(N) + b_{d}(N)) \big) - \clzn (\overline{\mathbf{x}})} \\
& - \modulo{ \clzn(\widehat{T}^N\widehat{\mathbf{x}}) - \clzn(\widehat{T}^N\overline{\mathbf{x}}) } \cdot \modulo{e \big(l(\widehat{\mathbf{x}} \cdot \widehat{\mathbf{a}}_{d}(N) + b_{d}(N)) \big)} - \modulo{ \clzn (\widehat{\mathbf{x}})  - \clzn (\overline{\mathbf{x}}) }\\
  \geq & \modulo{c_2 e(l \widehat{\mathbf{x}} \cdot \widehat{\mathbf{a}}_{d}(N))-c_1} -  \frac{1}{2}.
\end{split}
\end{equation*}
Call $\theta_1, \theta_2$ the argument of $c_1,c_2 \in \C$ respectively; fix $\theta \in (0, \frac{\pi}{2})$. If $r \in \R$ is such that $\theta_2 + 2 \pi r \notin [\theta_1-\theta, \theta_1+ \theta] + 2 \pi \Z$, then $\modulo{c_2e(r)-c_1} > \modulo{c_1}\sin \theta$, see Figure \ref{fig:1}.

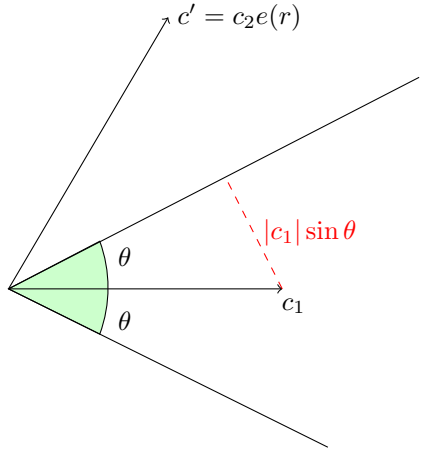
\begin{figure}[h!]
\centering
\begin{tikzpicture}[scale=3]
\clip (-0.2,-0.8) rectangle (2.5,1.3);
\draw [->] (0,0) -- (0.7,1.2) node [anchor=west] {$c\rq{} = c_2 e(r)$};
\filldraw[fill=green!20!white] (0,0) -- (0.4,-0.2) arc (-20:20:6mm)  -- cycle;
\draw [->] (0,0) -- (1.2,0);
\draw (1.25,0) node [anchor=north] {$c_1$};
\draw (0,0) -- (1.4,-0.7);
\draw (0,0) -- (1.8,0.935);
\draw (0.44,0.14) node [anchor=west] {$\theta$};
\draw (0.44,-0.14) node [anchor=west] {$\theta$};
\draw[dashed, red] (1.2,0) -- (0.95,0.495) node [pos=0.5,right] {\textcolor{red}{$\modulo{c_1} \sin \theta$}};
\end{tikzpicture}
\caption{ Any point $c\rq{} \in \C$ outside the cone of 1/2-angle $\theta$ about the line $\R c_1$ has distance from $c_1$ larger than the distance of $c_1$ from the boundary of the cone.} 
\label{fig:1}
\end{figure}

Thus, in our case, $\modulo{c_2 e(l\widehat{\mathbf{x}} \cdot \widehat{\mathbf{a}}_{d}(N))  - c_1} \leq \modulo{c_1}\sin \theta$ implies $\theta_2 + (2 \pi l) \widehat{\mathbf{x}} \cdot \widehat{\mathbf{a}}_{d}(N) \in [\theta_1-\theta, \theta_1+\theta] + 2 \pi \Z$; in particular, $l \widehat{\mathbf{x}} \cdot \widehat{\mathbf{a}}_{d}(N) $ belongs to an interval mod $\Z$ of size $\theta / \pi$. The level sets of the linear functional $\widehat{\mathbf{x}} \mapsto (2 \pi l) \widehat{\mathbf{x}} \cdot \widehat{\mathbf{a}}_{d}(N)$ are affine $(d-2)$-dimensional sets orthogonal to $\widehat{\mathbf{a}}_{d}(N)$ and hence parallel to a face $F$ of $Q$, see Figure \ref{fig:2}. 
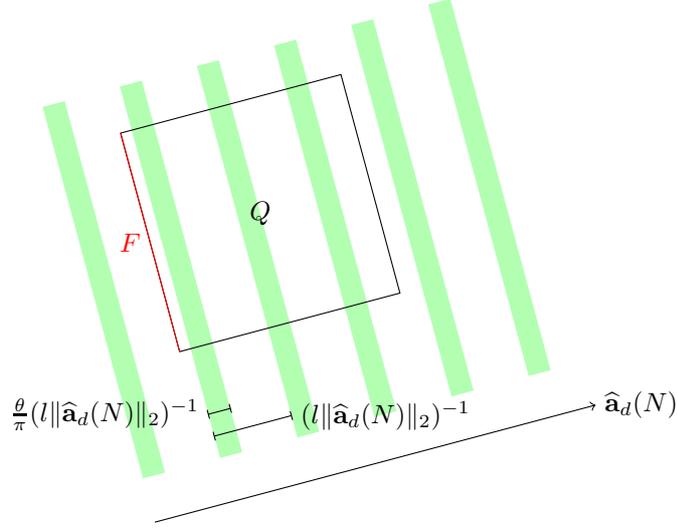
\begin{figure}[h!]
\centering
\begin{tikzpicture}[scale=3]
\fill[green!30!white] [rotate=15] (0,-0.2) rectangle (0.1,1.5);
\fill[green!30!white] [rotate=15] (0.35,-0.2) rectangle (0.45,1.5);
\fill[green!30!white] [rotate=15] (0.7,-0.2) rectangle (0.8,1.5);
\fill[green!30!white] [rotate=15] (1.05,-0.2) rectangle (1.15,1.5);
\fill[green!30!white] [rotate=15] (1.4,-0.2) rectangle (1.5,1.5);
\fill[green!30!white] [rotate=15] (1.75,-0.2) rectangle (1.85,1.5);
\draw [rotate=15] (0.3,0.3) rectangle (1.3,1.3);
\draw[white] [rotate=15] (0.8,0.8) -- (0.8,0.8) node {\textcolor{black}{$Q$}};
\draw[red] [rotate=15] (0.3,0.3) -- (0.3,1.3) node[pos= 0.5, left] {$F$};
\draw [rotate=15] [->] (0,-0.4) -- (2,-0.4) node[anchor=west] {$\widehat{\mathbf{a}}_{d}(N)$};
\draw [rotate=15] (0.35,0) -- (0.45,0) node[pos= 0, left] {$\frac{\theta}{\pi} (l \norma{\widehat{\mathbf{a}}_{d}(N)}_2)^{-1}$};
\draw [rotate=15] (0.45,-0.02) -- (0.45,0.02);
\draw [rotate=15] (0.35,-0.02) -- (0.35,0.02);
\draw [rotate=15] (0.35,-0.1) -- (0.7,-0.1) node[pos= 1, right] {$ (l \norma{\widehat{\mathbf{a}}_{d}(N)}_2)^{-1}$};
\draw [rotate=15] (0.7,-0.12) -- (0.7,-0.08);
\draw [rotate=15] (0.35,-0.12) -- (0.35,-0.08);
\end{tikzpicture}
\caption{In color, the set of $\mathbf{x}$ such that $\theta_2 + 2 \pi l \widehat{\mathbf{x}} \cdot \widehat{\mathbf{a}}_{d}\in [\theta_1-\theta, \theta_1+\theta] + 2 \pi \Z$.}
\label{fig:2}
\end{figure}

Therefore,
\begin{multline*}
\misura \Big(\widehat{\mathbf{x}}\in Q : \theta_2 + (2 \pi l) \widehat{\mathbf{x}} \cdot \widehat{\mathbf{a}}_{d}(N) \in [\theta_1-\theta, \theta_1+\theta] + 2 \pi \Z \Big) \leq \misura(F) \frac{\theta}{\pi} \left( L + \frac{1}{l  \norma{ \widehat{\mathbf{a}}_{d}(N) }_2} \right).
\end{multline*}
By Remark \ref{rk:one}, $\norma{ \widehat{\mathbf{a}}_{d}(N)}_2= O(N^k)$; since $L=O(1/(N^{k-1}\log N))$, we get
\begin{equation*}
\begin{split}
 \misura \left(\widehat{\mathbf{x}}\in Q : \modulo{c_2 e(l \widehat{\mathbf{x}} \cdot \widehat{\mathbf{a}}_{d}(N))  - c_1} \leq \modulo{c_1}\sin \theta \right) &
\leq \frac{\theta}{\pi} \misura (Q)  \left(1 + \frac{1}{l  \norma{ \widehat{\mathbf{a}}_{d}(N) }_2 L } \right) \\
& = \frac{\theta}{\pi} \misura (Q)  \left(1 + O\left( \frac{\log N }{N } \right) \right).
\end{split}
\end{equation*}
On the complement of this set, 
\begin{equation*}
\begin{split}
\modulo{\clznn(\widehat{\mathbf{x}})} &\geq  \modulo{c_2 e(l \widehat{\mathbf{x}} \cdot \widehat{\mathbf{a}}_{d}(N) )  - c_1}-\frac{1}{2} > \modulo{c_1}\sin \theta-\frac{1}{2} \geq \modulo{\clzn(\widehat{\mathbf{x}})}\sin \theta-\frac{3}{4};
\end{split}
\end{equation*}
hence
\begin{equation*}
\begin{split}
\limsup_{N \to \infty} \misura \left( \modulo{\clznn(\widehat{\mathbf{x}})}  \leq \modulo{\clzn(\widehat{\mathbf{x}})}\sin \theta-\frac{3}{4} \right) & \leq  \limsup_{N \to \infty} \sum_{Q \subset \Sigma} \left(1 + O\left( \frac{\log N}{N} \right)\right) \frac{\theta}{\pi} \misura(Q)  \\
 & = \limsup_{N \to \infty} \frac{\theta}{\pi} \left(1 + O\left( \frac{\log N}{N} \right)\right) = \frac{\theta}{\pi}.
\end{split}
\end{equation*}
We have obtained an estimate of the measure of the set where the coefficients $\clznn$ are small compared to $\clzn$; outside this set we can estimate $\modulo{S_n(\Psi^{\perp}) \circ T^N- S_n(\Psi^{\perp})}$ thanks to the hypothesis on $\modulo{S_n(\Psi^{\perp})}$ as follows.

\medskip

Let us add all these estimates as $0 < \modulo{l} \leq m$, where we recall $m$ is the degree of the trigonometric polynomial $\Psi^{\perp}$. Choose $C\rq{} \geq 9m^2$; pick $\theta \in (0, \frac{\pi}{2})$ such that $1/{\sqrt{C\rq{}}} \leq \sin \theta \leq {\sqrt{2/  C\rq{}}}$. Clearly, ${\theta}/{\pi} < \sin (\theta/2) = \sqrt{(1- \cos \theta) / 2} \leq (\sin \theta) / \sqrt{2} \leq1/{\sqrt{C\rq{}}}$. Outside a set of measure at most $2m({\theta} / {\pi}) \leq {2m}/{\sqrt{C\rq{}}}$, we have
\begin{equation*}
\begin{split}
\sum_{0 < \modulo{l} \leq m} \modulo{\clznn(\widehat{\mathbf{x}})} &\geq \sum_{0 < \modulo{l} \leq m} \modulo{\clzn(\widehat{\mathbf{x}})} \sin \theta - \frac{6m}{4} 
\geq \modulo{\sum_{0 < \modulo{l} \leq m}\clzn(\widehat{\mathbf{x}}) e(lx_d) }\frac{1}{\sqrt{C\rq{}}} - \frac{\sqrt{C\rq{}}}{2}.
\end{split}
\end{equation*}
We apply the following result.
\begin{lemma}[{\cite[Lemma 4]{avila:heisenberg}}]\label{lemmaavil}
For each $m \geq 1$ and for any norm $\norma{\cdot}_m$ on $\mathbb{C}^{2m}$, there exists constants $D_m$ and $d_m>0$ such that, if $\mathbf{c} = (c_{-m}, \dots, c_{-1},c_1, \dots, c_m)  \in \mathbb{C}^{2m}$ has unit norm $\norma{\mathbf{c} }_m = 1$, then for every $\delta >0$,
$$
\misura \left( \modulo{ \sum_{0 < \modulo{l} \leq m} c_l e(lx)} < \delta \right) < D_m\delta^{d_m}.
$$ 
\end{lemma}

Hence, in our case, there exist constants $D_m,d_m >0$ such that for every $\delta>0$ and for fixed $ \widehat{\mathbf{x}} \in \T^{d-1}$ the measure of the set of $x_d \in \mathbb{T}$ where $\modulo{(S_n(\Psi^{\perp}) \circ T^N-S_n(\Psi^{\perp}))(\widehat{\mathbf{x}},x_d)} < \delta \sum_{0 < \modulo{l} \leq m}  \modulo{\clznn(\widehat{\mathbf{x}})}$ is less than $D_m \delta^{d_m}$. 
By Fubini\rq{}s Theorem, choosing $\delta = {4C}/{\sqrt{C\rq{}}}$, outside a subset of $\T^d$ of measure less than $D_m \delta^{d_m}$ the following estimate holds:
\begin{equation*}
\modulo{(S_n(\Psi^{\perp}) \circ T^N- S_n(\Psi^{\perp}) )(\mathbf{x})} \geq \frac{4C}{\sqrt{C\rq{}}} \sum_{0 < \modulo{l} \leq m} \modulo{\clznn(\widehat{\mathbf{x}})}.
\end{equation*}
Thus, on a set of measure at least $ 1- {2m}/{\sqrt{C\rq{}}} - D_m ({4C}/{\sqrt{C\rq{}}} )^{d_m}$, we have
\begin{equation*}
\begin{split}
\modulo{(S_n(\Psi^{\perp}) \circ T^N- S_n(\Psi^{\perp}) )(\mathbf{x})} &\geq \frac{4C}{\sqrt{C\rq{}}} \Bigg( \modulo{\sum_{0 < \modulo{l} \leq m}\clzn(\widehat{\mathbf{x}}) e(lx_d) }\frac{1}{\sqrt{C\rq{}}} - \frac{\sqrt{C\rq{}}}{2} \Bigg)\\
&  = \frac{4C}{\sqrt{C\rq{}}}\left(\frac{\modulo{S_n(\Psi^{\perp})(\mathbf{x})}}{\sqrt{C\rq{}}} - \frac{\sqrt{C\rq{}}}{2} \right).
\end{split}
\end{equation*}
Let us enlarge $C\rq{}$ if necessary and choose $\varepsilon\rq{}>0$ such that 
$$
 \frac{2m}{\sqrt{C\rq{}}} + D_m \left( \frac{4C}{\sqrt{C\rq{}}} \right)^{d_m} + 2 \varepsilon\rq{} < \varepsilon.
$$
Let $n \geq 1$ such that $\modulo{S_n(\Psi^{\perp})} \geq C\rq{}$ up to a set of measure $\varepsilon\rq{}$, by Corollary \ref{th:cesaro}. Outside a set of measure less than $\varepsilon$, we conclude
$$
\modulo{(S_n(\Psi^{\perp}) \circ T^N- S_n(\Psi^{\perp}) )(\mathbf{x})} \geq 2C.
$$
\end{proof}


\subsection{Conclusion of the proof of Theorem \ref{th:stretchingroof}}

Lemma \ref{th:cesaro} implies that $\liminf_{n \to \infty} \misura_d (\modulo{S_n(\Psi^{\perp})}<C )=0$; let $L$ be the $\limsup$ and assume by contradiction that it is different from 0. Choose $\varepsilon >0$ and $\ell \geq 1$ such that 
$$
\frac{1}{\ell} + \frac{\ell+1}{2} \varepsilon < \frac{L}{2},
$$
and consider $C\rq{}>1$ and $\varepsilon\rq{}>0$ given by Lemma \ref{th:decoupling}. By Lemma \ref{th:cesaro}, there exists an arithmetic progression $\{ i\overline{n} \}_{i=1}^{\ell}$ of length $\ell$ such that $\misura_d(\modulo{S_{i\overline{n}}(\Psi^{\perp})} <C\rq{})< \varepsilon\rq{}$. 
By Lemma \ref{th:decoupling}, let $N_0(i) \geq 1$ be such that the conclusion \eqref{eq:concl} is satisfied with $n = i\overline{n}$; let $\overline{N_0}$ be the maximum of all $N_0(i)$  for $i=1, \dots, \ell$.
Choose $N  \geq \overline{N_0}$ such that $ \misura_d(\modulo{S_N(\Psi^{\perp})}<C) \geq \frac{L}{2}$. Since $T$ is measure-preserving, for $1 \leq j < i \leq \ell$ we get
\begin{equation*}
\begin{split}
&\misura_d(T^{-i\overline{n}}\{\modulo{S_N(\Psi^{\perp})}<C\} \cap T^{-j\overline{n}}\{\modulo{S_N(\Psi^{\perp})}<C\}) \\
& \qquad \qquad \leq \misura_d\left(\modulo{S_N(\Psi^{\perp}) \circ T^{i\overline{n}} - S_N(\Psi^{\perp}) \circ T^{j\overline{n}}}<2C\right) \\
& \qquad \qquad = \misura_d\left(\modulo{S_N(\Psi^{\perp}) \circ T^{(i-j)\overline{n}} - S_N(\Psi^{\perp}) }<2C\right),
\end{split}
\end{equation*}
which is less than $\varepsilon$ by Lemma \ref{th:decoupling}. Thus by the inclusion-exclusion principle, 
\begin{equation*}
\begin{split}
& \misura_d \left(\bigcup_{i=1}^{\ell} T^{-i\overline{n}} \left\{\modulo{S_n(\Psi^{\perp})}<C \right\}  \right) \geq \sum_{i=1}^{\ell} \misura_d \left(T^{-i\overline{n}}\left\{\modulo{S_n(\Psi^{\perp})}<C \right\}  \right)- \\
& \quad - \sum_{1 \leq j < i \leq \ell} \misura_d \left(T^{-i\overline{N}}\left\{\modulo{S_n(\Psi^{\perp})}<C \right\} \cap T^{-j\overline{n}}\left\{\modulo{S_N(\Psi^{\perp})}<C \right\} \right) \geq \ell \frac{L}{2} - \frac{\ell(\ell +1)}{2}\varepsilon.
\end{split}
\end{equation*}
This implies $L/2 \leq 1/\ell+\varepsilon(\ell +1)/2$, in contradiction with the initial choice of $\ell$ and $\varepsilon$. Thus $L=0$, which settles the proof.


\section{Proof of Theorem \ref{th:intro2}}\label{s:proof2}

In this section, we prove Theorem \ref{th:intro2} by reducing the problem to the setting of suspension flows over skew-translations as in Theorem \ref{th:intro1}. 
Recalling the definitions and notation of \S\ref{s:2.1}, let $F := F_d$ be a quasi-abelian filiform group, $M= \Lambda \backslash F$ a quasi-abelian filiform nilmanifold and $\{\varphi^{\mathbf{w}}_t\}_{t \in \R}$ a quasi-abelian filiform nilflow, where $\mathbf{w} = w_0 \mathbf{f}_0 + \cdots + w_d \mathbf{f}_d \in \mathfrak{f} = \mathfrak{f}_d$.

\subsection{Preliminaries}\label{s:2}
We first show that, although almost every quasi-abelian filiform nilflow is uniquely ergodic, they are not weak mixing; we remark that the same argument applies to general nilflows. 
Indeed, each nilflow has a factor which is isomorphic to a rotation on a torus and furthermore unique ergodicity of the latter is equivalent to unique ergodicity of the former, see, e.g., \cite[p.~344]{einsiedler:ergodic}.
\begin{lemma}
The abelianization $\abel \colon F \to F/F^{(2)}$ induces a factor of $(M, \{\varphi^{\mathbf{w}}_t\}_{t \in \R})$ which is isomorphic to a linear flow $\{\overline{\varphi}_t \}_{t \in \R}$ on $\T^2$.
\end{lemma}
\begin{proof}
We have the following diagram 
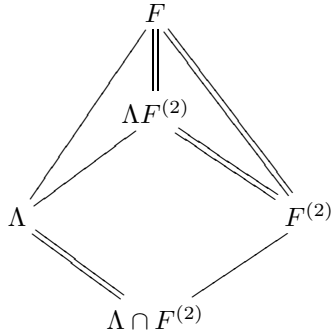
\begin{displaymath}
    \xymatrix{\ &  F\ar@{=}[d] \ar@{-}[ldd] \ar@{=}[rdd] &\ \\
              \ & \Lambda  F^{(2)}\ar@{-}[ld] \ar@{=}[rd] & \ \\
               \Lambda\ar@{=}[dr] & \ &  F^{(2)}\ar@{-}[dl] \\
               \ & \Lambda \cap  F^{(2)} }
\end{displaymath}
where the double line denotes a normal subgroup (indeed, $ F^{(2)}$ is characteristic and $\Lambda  F^{(2)}$ is normal since it contains $ F^{(2)}$).
The quotient $ F/  F^{(2)}$ is an abelian group which is isomorphic to $\R^2$; the abelianized lattice $\Lambda /  \Lambda \cap  F^{(2)} \simeq \Lambda   F^{(2)} /   F^{(2)}$ is isomorphic to $\Z^2$, so that the quotient $ F/ \Lambda   F^{(2)}$ is a $2$-dimensional torus and we obtain an exact sequence
$$
0 \to \Lambda \backslash \Lambda  F^{(2)} \to M \to \T^2 \to 0,
$$
which expresses $M$ as a bundle over the torus $\T^2$ with fibers isomorphic to $\Lambda \backslash \Lambda  F^{(2)}$.
The differential of the induced projection $\overline{\abel} \colon M \to \T^2$  on $M$ maps the vector field $\mathbf{w} = w_0 \mathbf{f}_0 + \cdots + w_d \mathbf{f}_d \in \mathfrak{f}$ to a vector field on $\T^2$, which gives the linear flow $\overline{\varphi}_t (x_0,x_1)= (x_0,x_1) + t (w_0,w_1)$.
\end{proof}
\begin{thm}[see, e.g., {\cite[p.~344]{einsiedler:ergodic}}]\label{thm:abel}
With the notation above, the following are equivalent:
\begin{itemize}
\item[(i)] $\{\varphi^{\mathbf{w}}_t\}_{t \in \R}$ is uniquely ergodic,
\item[(ii)] $\{\varphi^{\mathbf{w}}_t\}_{t \in \R}$ is ergodic with respect to the induced Haar measure $\mu$ on $M$,
\item[(iii)] $\{\overline{\varphi}_t \}_{t \in \R}$ is an irrational linear flow.
\end{itemize}
\end{thm}
In order to assure that the quasi-abelian filiform nilflow $\{\varphi^{\mathbf{w}}_t\}_{t \in \R}$ is uniquely ergodic, it is therefore sufficient to assume that $w_0$ and $w_1$ are rationally independent, which is a generic condition with respect to the Lebesgue measure on the Lie algebra $\mathfrak{f}$ of $F$.
For simplicity of notation, we will suppress the dependence on $\mathbf{w}$, writing $\varphi_t$ instead of $\varphi^{\mathbf{w}}_t$.

For any quasi-abelian filiform nilflow, there are nontrivial eigenfunctions for the Koopman operator $U_{\varphi_t}$ arising from the ones for the rotation $\overline{\varphi}_t$ on the torus $\T^2$: let $f \in L^2(\T^2)$ be a nontrivial eigenfunction for $U_{\overline{\varphi}_t}$, then the pull-back $(\overline{\abel})^{\ast} f \in L^2(M)$ is an eigenfunction for $U_{\varphi_t}$, since
$$
U_{\varphi_t}( (\overline{\abel})^{\ast}f) = (f \circ \overline{\abel}) \circ \varphi_t = f \circ \overline{\varphi}_t \circ \overline{\abel} = (U_{ \overline{\varphi}} f) \circ \overline{\abel} = f  \circ \overline{\abel}=(\overline{\abel})^{\ast}f.
$$
We are interested in the time-changes of $\{ \varphi_t \}_{t \in \R}$.
Let $\Sigma \subset M$ be a cross-section for the nilflow and let $\Psi \colon \Sigma \to \R_{>0}$ be the first return time function. Let $P \colon \Sigma \to \Sigma$  be the Poincaré map $P \mathbf{x}  = \varphi_{\Psi (\mathbf{x} )}(\mathbf{x} )$ for all $\mathbf{x}  \in \Sigma$. The nilflow $\{\varphi_t\}_{t \in \R}$ on $M$ is isomorphic to the suspension flow over $(\Sigma, P)$ with roof function $\Psi$.
For any time-change, the cross-section $\Sigma$ and the Poincaré map $P$ remain the same; on the other hand, one can check that the first return time map $\Psi^{\alpha}$ for the new flow with infinitesimal generator $\alpha$ is given by
$$
\Psi^{\alpha}(\mathbf{x} ) = \int_0^{\Psi (\mathbf{x} )}(\alpha \circ \varphi_t)(\mathbf{x} ) \diff t.
$$
The time-change $\{ \varphi^{\alpha}_t \}_{t \in \R}$ is then isomorphic to the suspension flow $\{P_t^{\Psi^{\alpha}}\}_{t \in \R}$ over $(\Sigma,P)$ with roof function $\Psi^{\alpha}$.


We prove Theorem \ref{th:intro2} by choosing a cross-section $\Sigma$ for the nilflow $\{{\varphi}_t\}_{t \in \R}$ such that, in appropriate coordinates, $\Sigma \simeq \T^d$ and the Poincaré map is a skew-translation as in Theorem \ref{th:intro1}. Moreover, the first return time is constant for all points in $\Sigma$; see Lemma \ref{th:csfil} below.

\subsection{Exponential coordinates and lattices} 
Let us recall (see, e.g., \cite[Theorem 1.2.1]{corwin:nilpotent}) that for any connected, simply connected nilpotent Lie group $G$ the exponential map $\exp \colon \mathfrak{g} \to G$ is an analytic diffeomorphism and the following \newword{Baker-Campbell-Hausdorff formula} holds:
\begin{equation}\label{eq:BKH}
\exp(\mathbf{v}) \exp(\mathbf{w}) = \exp \left( \mathbf{v} + \mathbf{w} + \frac{1}{2} [\mathbf{v},\mathbf{w}] + \cdots \right) \text{\ \ \ for any\ }  \mathbf{v},\mathbf{w} \in \mathfrak{g}.
\end{equation} 
We can use the exponential map to transfer coordinates from $\mathfrak{f}$ to $F$, so that we can cover the group with a single chart. In these coordinates, usually called the \newword{exponential coordinates}, the multiplication law becomes the Baker-Campbell-Hausdorff (BCH) product $\exp(\mathbf{v} \ast \mathbf{w}) = \exp(\mathbf{v}) \exp(\mathbf{w})$.   
Therefore, we can safely identify $F \simeq (\R^{d+1}, \ast)$.

It is possible to characterize lattices in quasi-abelian filiform groups using exponential coordinates. 
It is well-known that, for any co-compact lattice $\Lambda$, one can choose coordinates so that $\Lambda \simeq \Z^{d+1}$ (see, e.g., \cite[Theorem 5.1.6]{corwin:nilpotent}). However, for completeness and for the reader's convenience, we present a proof that provides new coordinates via a Lie algebra automorphism, hence preserving the Lie brackets.

Let us first state an auxiliary lemma.
Denote by $\text{Ad}\colon F \to \text{GL}(\mathfrak{f})$ the adjoint representation and by $\mathfrak{ad} \colon \mathfrak{f} \to \mathfrak{gl}(\mathfrak{f})$ its differential.

\begin{lemma}\label{th:comm}
For any $\mathbf{v}, \mathbf{w} \in \mathfrak{f}$ we have that 
$$
(-\mathbf{w}) \ast \mathbf{v} \ast \mathbf{w} = \Bigg(\sum_{j =0}^{d-1} \frac{\mathfrak{ad}(\mathbf{w})^j}{j!}\Bigg) \mathbf{v} = \mathbf{v}+ [\mathbf{w},\mathbf{v}] +\frac{1}{2}[\mathbf{w},[\mathbf{w},\mathbf{v}]] + \cdots.
$$
In particular, if $\mathbf{v}$ and $\mathbf{w}$ commute with $[\mathbf{v}, \mathbf{w}]$, we have that $\exp ([\mathbf{v}, \mathbf{w}]) = [\exp(\mathbf{v}), \exp(\mathbf{w})]_F$.
\end{lemma}
\begin{proof}
We compute $(\text{Ad} \circ \exp(\mathbf{w}))(\mathbf{v}) = \exp(- \mathbf{w}) \mathbf{v} \exp(\mathbf{w})$. By the commutation rule $\text{Ad} \circ \exp = \exp \circ \mathfrak{ad}$, it equals
$$
(\exp \circ \mathfrak{ad}(\mathbf{w}))(\mathbf{v}) = \mathbf{v} + [\mathbf{w}, \mathbf{v}]+\frac{1}{2}[\mathbf{w},[\mathbf{w},\mathbf{v}]] + \cdots. 
$$
We remark that, since $F$ is $d$-step nilpotent, $\mathfrak{ad}(\mathbf{w})^j = 0$ if $j \geq d$. Applying $\exp$ to both sides, we conclude
$$
\exp(- \mathbf{w}) \exp(\mathbf{v}) \exp(\mathbf{w}) = \exp \Bigg[ \Bigg(\sum_{j =0}^{d-1} \frac{\mathfrak{ad}(\mathbf{w})^j}{j!}\Bigg) \mathbf{v} \Bigg].
$$

If $\mathbf{v}$ and $\mathbf{w}$ commute with $[\mathbf{v}, \mathbf{w}]$, we have explicitly 
$$
\exp(- \mathbf{w}) \exp(\mathbf{v}) \exp(\mathbf{w}) = \exp( \mathbf{v} + [\mathbf{w}, \mathbf{v}] ) = \exp(\mathbf{v}) \exp( [\mathbf{w}, \mathbf{v}] ),
$$
from which we get $\exp ([\mathbf{v}, \mathbf{w}]) = [\exp(\mathbf{v}), \exp(\mathbf{w})]_F$.
\end{proof}

If the integer $E_1$ divides $E_2$ we write $E_1\ |\ E_{2}$.
\begin{lemma}\label{th:latticeF}
Let $\Lambda \leq F$ be a co-compact lattice in the $d+1$-dimensional quasi-abelian filiform group $F=F_d$ equipped with the exponential coordinates. Then, there exist $1=E_1\ |\ E_{2}\ |\ \cdots\ |\ E_d \in \N$, with $i!\ |\ E_i$, such that, up to an automorphism of $F$,  
$$
\Lambda = \left\{ x\mathbf{f}_0 +  \sum_{i=1}^d \frac{y_i}{E_i} \mathbf{f}_{i} : x, y_i \in \Z \right\}.
$$
\end{lemma}
\begin{proof}
Let $\pi_{i}$ be the canonical projection of $F=F_d$ onto $F / F^{(i)}$. The image $\pi_2(\Lambda) \subset F / F^{(2)}$ is a lattice in $\R^2$, hence there exist $\mathbf{v}_0, \mathbf{v}_1 \in \Lambda$ such that $\pi_2(\mathbf{v}_0), \pi_2(\mathbf{v}_1)$ generate $\pi_2(\Lambda)$. We can suppose that the first component of $\mathbf{v}_0$ in the basis $\mathcal{F}_d = \{\mathbf{f}_0, \dots, \mathbf{f}_d\}$ is different from zero.

We first show that for every $1 \leq i \leq d$ there exists $\mathbf{v}_i \in \Lambda \cap F^{(i)} \setminus F^{(i+1)}  $. By induction, suppose there exists $\mathbf{v}_{i-1} \in \Lambda \cap F^{(i-1)} \setminus F^{(i)} $ for $i \geq 2$. Then, by Lemma \ref{th:comm},
$$
[\pi_{i+1}(\mathbf{v}_0), \pi_{i+1}(\mathbf{v}_{i-1}) ] = \pi_{i+1} ([\mathbf{v}_0, \mathbf{v}_{i-1}]) \in (\Lambda \cap F^{(i)}) / F^{(i+1)},
$$
since it belongs to the centre of $F / F^{(i+1)}$. It is also different from zero, as $\mathbf{v}_{i-1} \notin  F^{(i)} $.
Thus, there exists $\mathbf{v}_i \in \Lambda \cap F^{(i)} \setminus F^{(i+1)}$ such that $ \pi_{i+1} (\mathbf{v}_i)= \pi_{i+1}([\mathbf{v}_0, \mathbf{v}_{i-1}]) $, hence the claim.

If $d=1$, the group $F_1$ is abelian and isomorphic to $\R^2$ and the conclusion follows. 
Suppose $d \geq 2$ and let $\mathbf{v}_0, \mathbf{v}_1 \in \Lambda$ as above. Consider $\mathbf{v}_{d-1} \in \Lambda \cap F^{(d-1)}$; by Lemma \ref{th:comm}, we have $[\mathbf{v}_0, \mathbf{v}_{d-1}], [\mathbf{v}_1, \mathbf{v}_{d-1}] \in \Lambda \cap F^{(d)}$. The latter is isomorphic to a discrete subgroup of $\R$, thus the two vectors are rationally dependent. This implies that the first coordinate of $\mathbf{v}_0$ and $\mathbf{v}_1$ are rationally dependent. Up to replace $\mathbf{v}_1$ with a vector of the form $(-\mathbf{v}_0) \ast \cdots \ast (-\mathbf{v}_0) \ast \mathbf{v}_1 \ast \cdots \ast \mathbf{v}_1 \in \Lambda$, we can suppose that the first coordinate of $\mathbf{v}_1$ is zero.

Define $\ell \colon F \to F$ as the unique group automorphism such that $\ell(\mathbf{v}_0) = \mathbf{f}_0$ and $\ell(\mathbf{v}_1) = \mathbf{f}_1$. Then, $ \mathbf{f}_0$ and $ \mathbf{f}_1$ generate the projected lattice $\ell(\Lambda) / F^{(2)}$ and moreover, by Lemma \ref{th:comm}, $\ell (\Lambda)$ contains
$$
(- \mathbf{f}_1 ) \ast ( - \mathbf{f}_0) \ast  \mathbf{f}_1 \ast  \mathbf{f}_0 = -  \mathbf{f}_1 \ast \Bigg(  \mathbf{f}_1 + \sum_{i=2}^d \frac{1}{i!} \mathbf{f}_i \Bigg) = \sum_{i=2}^d \frac{1}{i!} \mathbf{f}_i.
$$
Inductively, by replacing $ \mathbf{f}_1$ above with $\sum_{i \geq 2} (i!)^{-1} \mathbf{f}_i$ and so on, it is easy to see that $\ell(\Lambda)$ contains the lattice generated by $\frac{1}{i!} \mathbf{f}_i $, hence
$$
\ell(\Lambda ) = \Z \times \frac{1}{E_1} \Z \times \cdots \times \frac{1}{E_d} \Z,
$$
for some integers $E_1 =1, E_2, \dots, E_d$ such that $i!\ |\ E_i$. 
Moreover, for all $1\leq i \leq d$, 
$$
\Big(- \frac{1}{E_i} \mathbf{f}_i \Big) \ast ( - \mathbf{f}_0) \ast  \Big( \frac{1}{E_i} \mathbf{f}_i \Big) \ast  \mathbf{f}_0 = \frac{1}{E_i} \mathbf{f}_{i+1} + \text{ terms in }F^{(i+2)},
$$
hence $E_i\ |\ E_{i+1}$.
\end{proof}
We consider the new basis $\mathcal{F}_d\rq{} = \{ \mathbf{f}_0\rq{}, \dots, \mathbf{f}_d\rq{} \}$, where $\mathbf{f}_0\rq{} = \mathbf{f}_0$ and $\mathbf{f}_{i}\rq{} = (1/E_i) \mathbf{f}_{i}$ for $i=1, \dots, d$. In this way, we have $\Lambda = (\Z^{d+1}, \ast) \leq F$ and the only nontrivial brackets are $[\mathbf{f}_0\rq{}, \mathbf{f}_i\rq{} ]= (E_{i+1}/E_i) \mathbf{f}_{i+1}\rq{}$.

\subsection{Reduction to suspension flows}

Let $\mathbf{w} = (w_0, \dots, w_d) \in \mathfrak{f}$ be a vector inducing a uniquely ergodic nilflow on $M = \Lambda \backslash F$; equivalently, by Theorem \ref{thm:abel}, such that $w_0 / w_1 \notin \mathbb{Q}$.
Define the smooth submanifold
$$
\Sigma = \{ \Lambda (0,x_1, \dots, x_d) : x_i \in \R,\ 1\leq i \leq d\}.
$$
Since the ideal generated by $\mathbf{f}_1, \dots, \mathbf{f}_{d}$ is abelian, the submanifold $\Sigma$ is isomorphic to a torus $\T^d$ via the map 
\begin{equation*}
\begin{split}
\varsigma \colon \bigslant{\R^{d}}{\Z^d}\text{\ \ \ \ \ \ \ \ \ \ \ } &\to \Sigma \\
\mathbf{x} = (x_1, \dots, x_d) & \mapsto \Lambda (0,x_1, \dots, x_d).
\end{split}
\end{equation*}
\begin{lemma}\label{th:csfil}
The first return time to $\Sigma$ is constant for any point of $\Sigma$; the Poincaré map $P \colon \Sigma \to \Sigma$ is given by
$$
P \circ \varsigma(\mathbf{x}) = \varsigma \left( \mathbf{x} A + \mathbf{b}\right)
$$
for some $\mathbf{b} \in \T^d$ and an upper triangular $d \times d$ matrix $A = (a_{i,j})$, with $a_{i,j} = E_j / (E_i \cdot (j-i)!)$ for $1\leq i\leq j \leq d$. 
\end{lemma}
\begin{proof}
Let $\varsigma(\mathbf{x}) = \Lambda (0,\mathbf{x}) \in \Sigma$. By definition, we have
$$
\varphi_{1/w_0}(\Lambda (0,\mathbf{x})) = \Lambda (0, x_1, \dots, x_d) \ast \Big( 1, \frac{w_1}{w_0}, \dots, \frac{w_d}{w_0}\Big).
$$
Since $\Lambda = \Lambda (-1,0, \dots, 0)$, by Lemma \ref{th:comm} we get
\begin{equation*}
\begin{split}
\varphi_{1/w_0}(\Lambda (0,\mathbf{x})) &= \Lambda (-1,0, \dots, 0) \ast (0, x_1, \dots, x_d) \ast \Big( 1, \frac{w_1}{w_0}, \dots, \frac{w_d}{w_0}\Big)\\
&=  \Lambda \Bigg( \sum_{j=0}^d \frac{\mathfrak{ad}(1,0, \dots, 0)^j}{j!} (0, \mathbf{x}) \Bigg) \ast (-1,0, \dots, 0) \ast \Big( 1, \frac{w_1}{w_0}, \dots, \frac{w_d}{w_0}\Big).
\end{split}
\end{equation*}
Therefore, defining $(0,b_1, \dots, b_d) = (-1,0, \dots, 0) \ast (1, w_1/w_0, \dots, w_d/w_0)$, we obtain
\begin{equation*}
\begin{split}
\varphi_{1/w_0}(\Lambda (0,\mathbf{x})) &= \Lambda \Bigg( 0, x_1, \dots, \sum_{i=0}^{j-1} \frac{1}{(j-i)!} \frac{E_j}{E_{i}} x_{i}, \dots, \Bigg) \ast ( 0, b_1, \dots, b_d)\\
&= \Lambda \Bigg( 0, x_1 + b_1, \dots, \sum_{i=0}^{j-1} \frac{1}{(j-i)!} \frac{E_j}{E_{i}} x_{i} + b_j, \dots \Bigg).
\end{split}
\end{equation*}
The set of return times to $\Sigma$ is a subset of the set of the return times of the projected linear flow on the abelianization $F/F^{(2)} \simeq \T^2$, which is $(1/w_0)\Z$.
The equation above shows that $1/w_0$ is indeed a return time, hence it is the first return time to $\Sigma$, and the Poincaré map is of the requested form.
\end{proof}

We showed that any uniquely ergodic nilflow $\{ \varphi_t \}_{t \in \R}$ is isomorphic to a suspension flow over a skew-translation $(\T^d,T)$ with constant roof function $\Psi \equiv 1$. As discussed in \S\ref{s:2}, given the infinitesimal generator $\alpha$ of a time-change $\{ \varphi^{\alpha}_t \}_{t \in \R}$, the new roof function $\Psi^{\alpha} = R(\alpha)$ is given by 
$$
\Psi^{\alpha}(\mathbf{x}) = \int_0^1 (\alpha \circ \varphi_t)(\mathbf{x}) \diff t.
$$
The map $R \colon \mathscr{C}^{\infty}(M) \to \mathscr{C}^{\infty}(\T^d)$, $R(\alpha) = \Psi^{\alpha} $ is linear, surjective and continuous w.r.t.~$\norma{\cdot}_{\infty}$, thus $R^{-1}(\mathscr{R})$ is a dense set of infinitesimal generators. Theorem \ref{th:intro2} now follows from Theorem~\ref{th:intro1}.


\section{Appendix: proof of Theorem \ref{th:ulcigrai}}\label{SU}

The proof of this result follows closely the argument by Avila, Forni and Ulcigrai in \cite{avila:heisenberg}: we outline the main ideas, referring the reader to the cited article for the details. We use the same notation as in \S\ref{ST1}.


\subsection{Shearing}
We briefly explain the shearing phenomenon that produces mixing; a similar mechanism was used by many authors in different contexts, see \cite{marcus:horocycle, sinai:mixing, fayad:mixing, ulcigrai:mixing, ravotti:lhf}.
We want to apply the following criterion, see \cite{fayad:mixing} and \cite[\S1.3.2]{ulcigrai:mixing} for details.
\begin{lemma}[Mixing Criterion]\label{th:mixingcriterion}
The suspension flow $\{T_t^{\Psi}\}_{t \in \R}$ is mixing if for any cube $Q= \prod_{i=1}^d [w_i, w_i\rq{}] \times [0,h]$, with $0 < h <\min \Psi$, any $\varepsilon >0$ and $\delta >0$ there exists $t_0 \geq 0$ such that for all $t \geq t_0$ there exists a measurable set $\widehat{X}(t) \subset \T^{d-1}$ and for each $\widehat{\mathbf{x}} = \pi(\mathbf{x}) \in \widehat{X}(t)$ there exists a partition $\ptx{m} $ into intervals $J \subset  \{\widehat{\mathbf{x}} \} \times \T$ such that 
\begin{equation}\label{eq:mixingset}
\misura_d \big(\T^d \setminus \cup_{\widehat{\mathbf{x}}  \in \widehat{X}(t)} \ptx{m} \big) \leq \delta,
\end{equation} 
and for all $\widehat{\mathbf{x}}  \in \widehat{X}(t)$ and all $J= \{ \widehat{\mathbf{x}} \} \times   [a,b]\in \ptx{m}$,
\begin{equation}\label{eq:mixingestimate}
\misura_1 \big( J \cap T^{\Psi}_{-t}(Q) \big) \geq (1-\varepsilon)(b-a)\misura_d(Q).
\end{equation}
\end{lemma}
In order to apply Lemma \ref{th:mixingcriterion}, we will construct a partition of intervals $J$ in the $x_d$-direction most of which becomes \newword{sheared} for sufficiently large $t$. More precisely, for any $J = \{ \widehat{\mathbf{x}} \} \times   [a,b]$, we define the \newword{stretch} of $S_n(\Psi)$ over $J$ as
$$
\Delta S_n(\Psi)(J) = \max_{\mathbf{x} \in J} S_n(\Psi)(\mathbf{x}) - \min_{\mathbf{x} \in J} S_n(\Psi)(\mathbf{x}).
$$  
We will prove that, for a set of intervals $J$ whose measure is large in $\T^d$, the stretch $\Delta S_n(\Psi)(J)$ is large for all $n$ of the form $n = n_t(\mathbf{x})$ for some $\mathbf{x} \in J$ and large $t$. This would imply that the image of $J$ after time $t$ can be written as the union of curves $\gamma_i = T^{\Psi}_t(J_i)$, for subintervals $J_i \subset J$, which project over intervals in the $x_d$-direction and on which the derivative $\partial_d S_n(\Psi)$ of $S_n(\Psi)$ w.r.t.~$x_d$ is large. The base points of these curves, i.e.~the intersections $\gamma_i \cap \T^d \times \{0\}$, shadow with good approximation an orbit under $T$, hence, by unique ergodicity, are uniformely distributed in $\T^d$; this leads to the mixing estimate.


\subsection{Stretch of Birkhoff sums for continuous time}
                           
Recall that for $\mathbf{x}=(\widehat{\mathbf{x}} ,x_d) \in \T^d$ we denote $n_t(\mathbf{x}) = \max \{ n : S_n(\Psi)(\mathbf{x}) \leq t\}$; let
\begin{equation*}
\ntyx = \min \{n_t(\widehat{\mathbf{x}} ,x_d) : x_d \in \T \}. \\ 
\end{equation*}
The following lemma ensures that the Birkhoff sums $S_n(\Psi^{\perp})$ grow in measure not only as $n$ tends to infinity (see Theorem \ref{th:stretchingroof}), but also when $t$ tends to infinity. The proof uses the assumption that $\psi$ is smoothly cohomologous to a constant. 
\begin{lemma}
For all $C>1$, let 
$$
\widehat{X}(t,C) = \left\{ \widehat{\mathbf{x}}  \in \T^{d-1} : \text{there exists $x_d \in \T$ s.t.} \modulo{S_{\ntyx}(\Psi^{\perp})( \widehat{\mathbf{x}} ,x_d)} >C \right\}.
$$
Then
$$
\lim_{t \to \infty} \misura \left( \T^{d-1} \setminus \widehat{X}(t,C) \right) = 0.
$$
\end{lemma}
\begin{proof}
Let us assume by contradiction that there exist $C>1$, $\delta >0$ and an increasing sequence $\{t_j\}_{j \in \N}$, with $t_j \to \infty$, such that $\misura \left( \T^{d-1} \setminus \widehat{X}(t_j,C) \right) \geq \delta$ for all $j \in \N$. 
If $\widehat{\mathbf{x}} \notin \widehat{X}(t_j,C)$, for all $x_d \in \T$ we have $|S_{\ntjx}(\Psi^{\perp})(\widehat{\mathbf{x}} ,x_d)| \leq C$; thus, by Fubini\rq{}s Theorem, 
$$
\misura \left\{ \mathbf{x} \in \T^d : \modulo{S_{\ntjx}(\Psi^{\perp})(\widehat{\mathbf{x}} ,x_d)} \leq C \right\} \geq \misura \left( \T^{d-1} \setminus X(t_j,C) \right) \geq \delta.
$$

As we want to get a contradiction with Theorem \ref{th:stretchingroof}, we look for a sequence $\{\ntjx\}_{j \in \N}$ not depending on the point $\widehat{\mathbf{x}}$. Since $\psi$ is smoothly cohomologous to the constant $\int \Psi$, there exists a smooth function $u \colon \Sigma \to \R$ such that $\psi - \int \Psi = u \circ T - u$. Let $\mathbf{y}$ be the point in $\T^d$ for which $\ntjx = n_{t_j}(\mathbf{y})$. We have
\begin{equation*}
\begin{split}
S_{\ntjx}(\Psi)(\mathbf{y}) &= S_{\ntjx}(\Psi^{\perp})(\mathbf{y}) + S_{\ntjx}(\psi)(\mathbf{y}) \\
& = S_{\ntjx}(\Psi^{\perp})(\mathbf{y})  + u(T^{\ntjx}\mathbf{y})  - u(\mathbf{y}) + \ntjx \cdot \int_{\T^d} \Psi\diff \misura_d.
\end{split}
\end{equation*}
Let $\overline{u}$ and $\overline{\Psi}$ be the maximum of $\modulo{u}$ and of $\Psi$ over $\T^d$. Since, by definition, $t_j - \overline{\Psi} \leq S_{\ntjx}(\Psi)(\mathbf{y}) = S_{n_{t_j}(\mathbf{y})}(\Psi)(\mathbf{y}) \leq t_j$, from the previous equation it follows that for all $\widehat{\mathbf{x}} \notin \widehat{X}(t_j,C)$, 
$$
t_j - \overline{\Psi} - C - 2 \overline{u} \leq \ntjx \cdot \int_{\T^d} \Psi\diff \misura_d \leq t_j + C + 2\overline{u}.
$$
In particular, there exists a constant $K$ such that for all $t_j$ there are at most $K$ possible values of $\ntjx$. Therefore, there exists a sequence $n_j = \underline{n}_{t_j}(\mathbf{x}_j)$ such that  $\misura( |S_{n_j}(\Psi^{\perp})(\widehat{\mathbf{x}},x_d)| \leq C) \geq \delta / K$, so that $\limsup_{n \to \infty}\misura(|S_n(\Psi^{\perp})|\leq C) \geq \delta / K >0$, in contradiction with Theorem \ref{th:stretchingroof}.
\end{proof}

\begin{remark}
Straightforward computations show that $\partial_d(S_n(\Psi)) = S_n(\partial_d\Psi) = S_n(\partial_d \Psi^{\perp})$ and $\partial_d^2(S_n(\Psi)) = S_n(\partial_d^2\Psi) = S_n(\partial_d^2 \Psi^{\perp})$ for all $n \geq 1$. Indeed, $\partial_d \Psi = \partial_d \Psi^{\perp}$, since $\psi = \int \Psi \diff x_d$ does not depend on $x_d$; moreover, as a map in the $x_d$-coordinate, $T^i$ is a translation for all $i \geq 1$, hence $\partial_d (\Psi \circ T^i) = \partial_d\Psi \circ T^i$.
\end{remark}


\subsection{The Mixing Criterion}

Let $Q= \prod_{i=1}^d [w_i, w_i\rq{}] \times [0,h]$ be a given cube. 
Choose $\delta_0 \in (0,1)$ such that $(1-\delta_0)(1 - D\rq{}\delta_0^{d\rq{}}-m\delta_0) \geq 1-\delta$, where $D\rq{}, d\rq{}$ are given by Lemma \ref{lemmaavil} w.r.t.~$\mnorm{\cdot}$, with $\mnorm{\sum_{\modulo{j}\leq m} \alpha_j e(jz)} = \max_j \modulo{\alpha_j}$. Let $\varepsilon_0, N_0, C_0$ be chosen appropriately as in \cite[\S4.5]{avila:heisenberg}; let $\chi$ be a continuous function such that
\begin{equation}\label{eq:defofchi}
\chi(\mathbf{x}) = 
\begin{cases}
1 & \text{if } \mathbf{x} \in \prod_{i=1}^{d-1} [w_i, w_i\rq{}] \times [w_d+\varepsilon_0(w_d'-w_d),w_d'-\varepsilon_0(w_d'-w_d)],\\
0 & \text{if } \mathbf{x} \notin \prod_{i=1}^{d-1} [w_i, w_i\rq{}] \times [w_d+\varepsilon_0/2(w_d'-w_d),w_d'-\varepsilon_0/2(w_d'-w_d)].
\end{cases}
\end{equation}
Finally, let $t_0>0$ be such that for all $t \geq t_0$ we have $\misura(\T^{d-1} \setminus \widehat{X}(t,C_0))\leq \delta_0$. Set $\widehat{X}(t,C_0) = \widehat{X}(t)$.

We recall \eqref{eq:psienne},
$$
S_n(\Psi^{\perp})(\widehat{\mathbf{x}} ,x_d) = \sum_{0 < \modulo{l} \leq m} c_{l,n}(\widehat{\mathbf{x}} )  e(l x_d ),
$$
and denote $\clznp (\widehat{\mathbf{x}} )= 2 \pi i l \clzn (\widehat{\mathbf{x}} )$ so that we can write
\begin{equation}\label{eq:f1}
S_n(\partial_d \Psi)(\widehat{\mathbf{x}} ,x_d) = S_n(\partial_d \Psi^{\perp})(\widehat{\mathbf{x}} ,x_d)= \sum_{0< \modulo{l} \leq m} \clznp(\widehat{\mathbf{x}} )e(l x_d).
\end{equation}
Let 
$$
\ptx{0} = \left\{ (\widehat{\mathbf{x}} ,x_d) \in \{\widehat{\mathbf{x}} \} \times \T : \modulo{S_{\ntyx}(\partial_d \Psi)(\widehat{\mathbf{x}} ,x_d)} \geq \delta_0 \mnorm{S_{\ntyx}(\partial_d \Psi)} \right\},
$$
which is a union of intervals in the $x_d$-coordinate, since, for fixed $n$, $S_n(\Psi^{\perp})(\widehat{\mathbf{x}} , \cdot)$ is a polynomial in $x_d$ of degree $m$. Let $\ptx{1}$ be the partial partition obtained by discarding form $\ptx{0}$ all intervals of length less than $\delta_0$.
By Lemma \ref{lemmaavil}, we have $\misura( \ptx{0}) \geq 1 - D\rq{}\delta_0^{d\rq{}}$. Again, since $S_{\ntyx}(\partial_d \Psi)(\widehat{\mathbf{x}} ,x_d)$ is a trigonometric polynomial of degree $m$, there are at most $2m$ points in each level set; therefore $\ptx{1}$ is obtained from $\ptx{0}$ by removing at most $m$ intervals of length smaller than $\delta_0$. The size of the partial partition $\ptx{1}$ satisfies
$$
\misura(\ptx{1}) \geq 1 - D\rq{}\delta_0^{d\rq{}}-m\delta_0,
$$
thus, by Fubini\rq{}s Theorem,
\begin{equation}\label{eq:f2}
\misura_{d-1} \left( \bigcup_{\widehat{\mathbf{x}}  \in \widehat{X}(t)} \ptx{1} \right) \geq (1-\delta_0)(1 - D\rq{}\delta_0^{d\rq{}}-m\delta_0) \geq 1-\delta,
\end{equation}
by the choice of $\delta_0$.

The following lemma ensures that on each element of the partition the stretch is large enough. For all $I \in \ptx{1}$, denote by $\underline{n}_t(I) = \min_{x_d} n_t(\widehat{\mathbf{x}} ,x_d)$, $\overline{n}_t(I) = \max_{x_d} n_t(\widehat{\mathbf{x}} ,x_d)$, and $\Delta n_t(I) =  \overline{n}_t(I) - \underline{n}_t(I) +1$.
\begin{lemma}
For all $I \in \ptx{1}$ we have that 
\begin{align}
& \modulo{S_{\ntyx}(\partial_d \Psi)(\widehat{\mathbf{x}} ,x_d)} \geq \frac{\pi \delta_0}{m}C_0, \text{\ \ \ for all } (\widehat{\mathbf{x}} ,x_d) \in I; \label{eqq}\\
& \modulo{S_{\ntyx}(\partial_d \Psi)(\widehat{\mathbf{x}} ,x_d)} \geq \frac{\delta_0}{2m} \modulo{S_{\ntyx}(\partial_d \Psi)(\widehat{\mathbf{x}} ,x_d')} , \text{\ \ \ for all } (\widehat{\mathbf{x}} ,x_d), (\widehat{\mathbf{x}} ,x_d') \in I. \label{eqqq}
\end{align}
\end{lemma}
\begin{proof}
From \eqref{eq:f1}, for all $(\widehat{\mathbf{x}},x_d) \in \T^d$ and $n \geq 1$ we have that 
$$
\modulo{S_n(\partial_d \Psi)(\widehat{\mathbf{x}} ,x_d)} \leq (2 m) \max_{0<\modulo{l} \leq m} \modulo{c_{l,n}'(\widehat{\mathbf{x}})};
$$
hence, from the definition of $\ptx{1} \subset \ptx{0}$,
$$
\min_{x_d \in \T}\modulo{S_n(\partial_d \Psi)(\widehat{\mathbf{x}} ,x_d)} \geq \frac{\delta_0}{2m} \max_{x_d \in \T}\modulo{S_n(\partial_d \Psi)(\widehat{\mathbf{x}} ,x_d)}.
$$
This proves \eqref{eqq}. Moreover, by definition of $\widehat{X}(t)$, there exists $\mathbf{x} = (\widehat{\mathbf{x}} ,x_d)$ for which $| S_{\ntyx}(\partial_d \Psi)(\mathbf{x}) | \geq C_0$. Thus, $\max_{0<\modulo{l} \leq m} | c_{l,n}(\widehat{\mathbf{x}}) | \geq C_0/(2m)$, so that $\max_{0<\modulo{l} \leq m} | c_{l,n}'(\widehat{\mathbf{x}}) | \geq 2 \pi \max_{0<\modulo{l} \leq m} | c_{l,n}(\widehat{\mathbf{x}}) | \geq \pi C_0/m$. We conclude \eqref{eqqq} from the definition of $\ptx{0}$.
\end{proof}
From the previous estimates, it is possible to deduce the following properties; for the proof we refer to \cite[Lemmas 11,12]{avila:heisenberg}.
\begin{lemma}[{\cite[Lemmas 11,12]{avila:heisenberg}}]
For all $I \in \ptx{1}$ and for all $\underline{n}_t(I) \leq n \leq \overline{n}_t(I)$, we have 
$$
\frac{1}{2} \modulo{S_{\ntyx}(\partial_d \Psi)(\widehat{\mathbf{x}} ,x_d)} \leq \modulo{S_{n}(\partial_d \Psi)(\widehat{\mathbf{x}} ,x_d)} \leq \frac{3}{2} \modulo{S_{\ntyx}(\partial_d \Psi)(\widehat{\mathbf{x}} ,x_d)} \text{\ \ \ for all } (\widehat{\mathbf{x}} ,x_d) \in I.
$$
Moreover, the function $x_d \mapsto n_t(\widehat{\mathbf{x}} ,x_d)$ is monotone and $\Delta n_t(I) \geq \pi \delta_0^2C_0 / (2m \min \Psi)$.
\end{lemma}
Let us subdivide each interval $I \in \ptx{1}$ into $\Delta n_t(I)$ subintervals on which $x_d \mapsto n_t(\widehat{\mathbf{x}} ,x_d)$ is locally constant and let us group them into $N_t(I) + 1$ groups, the first $N_t(I)$ of which made by $N_t(I)$ consecutive intervals, where $N_t(I) = \lfloor \sqrt{\Delta n_t(I)} \rfloor$. Denote by $\ptx{m}$ the partition into intervals $J$ obtained in this way. The estimate on the total measure \eqref{eq:f2} still holds. Moreover, each $J \in \ptx{m}$ satisfies the following properties, which can be proved using the estimates on the stretch and on the size of the intervals, see \cite[Lemma 13]{avila:heisenberg}. 
\begin{lemma}[{\cite[Lemma 13]{avila:heisenberg}}]
For each $J \in \ptx{m}$, for all $(\widehat{\mathbf{x}} ,x_d), (\widehat{\mathbf{x}} ,x_d') \in J$ and all $\underline{n}_t(J) \leq n \leq \overline{n}_t(J)$ we have
\begin{align}
& \modulo{\frac{\Delta n_t(J)}{\Delta S_{\underline{n}_t(J) }(J)} - 1} \leq \varepsilon_0; \label{eq:final1} \\
& \frac{1}{\Delta n_t(J)} \sum_{n = \underline{n}_t(J)}^{\overline{n}_t(J)} \chi \circ T^n(\widehat{\mathbf{x}} ,x_d) \geq (1-\varepsilon_0)^2 \prod_{i=1}^d (w_i'-w_i); \label{eq:final2} \\
& \misura_1(J) \leq \frac{w_d'-w_d}{2}\varepsilon_0; \label{eq:final3}\\
& \modulo{\frac{\Delta S_{\underline{n}_t(J) }(J)}{\Delta S_{n}(J)} -1} \leq \varepsilon_0. \label{eq:final4}
\end{align}
Moreover, denoting $J_n^h = \{ (\widehat{\mathbf{x}} ,x_d) \in J : t-h  < S_n(\Psi) (\widehat{\mathbf{x}} ,x_d) \leq t\}$, we have
\begin{equation}\label{eq:final5}
\modulo{\frac{\Delta S_{n }(J) \misura_1 (J_n^h)}{\misura_1(J) h} -1} \leq \varepsilon_0.
\end{equation}
\end{lemma}  
It remains to prove \eqref{eq:mixingestimate} of the Mixing Criterion. 
By definition, $J_n^h$ is the set of points in $J$ that after time $t$ undergo exactly $n$ iterations of $T$ (recall that $h < \min \Psi$) and are mapped inside $\T^d \times [0,h]$. In particular, for different values of $n$, they are all disjoint. If, for $J = \{\widehat{\mathbf{x}} \} \times (x_d',x_d'') \in \ptx{m}$, we have that $\chi(T^n (\widehat{\mathbf{x}} ,x_d') ) >0$, by the estimate \eqref{eq:final3} on the size of $J$ and the definition of $\chi$ \eqref{eq:defofchi}, it follows that $T^n (\widehat{\mathbf{x}} ,x_d) \in \prod_i [w_i,w_i']$ for all $ (\widehat{\mathbf{x}} ,x_d) \in J_n^h$ and thus $T_t^{\Psi}(\widehat{\mathbf{x}} ,x_d) \in Q$. We deduce that 
$$
\misura_1 ( J \cap T^{\Psi}_{-t}(Q)) \geq \sum_{n=\underline{n}_t(J)}^{\overline{n}_t(J)} \chi \circ T^n(\widehat{\mathbf{x}} ,x_d') \misura_1(J_n^h).
$$
Using \eqref{eq:final1}, \eqref{eq:final2}, \eqref{eq:final4} and \eqref{eq:final5}, we conclude
\begin{equation*}
\begin{split}
& \sum_{n=\underline{n}_t(J)}^{\overline{n}_t(J)} \chi \circ T^n(\widehat{\mathbf{x}} ,x_d') \misura_1(J_n^h) \\
& = \frac{1}{\Delta n_t(J)} \sum_{n = \underline{n}_t(J)}^{\overline{n}_t(J)} \chi \circ T^n(\widehat{\mathbf{x}} ,x_d) \frac{\Delta n_t(J)}{\Delta S_{\underline{n}_t(J) }(J)} \frac{\Delta S_{\underline{n}_t(J) }(J)}{\Delta S_{n}(J)} \frac{\Delta S_{n }(J) \misura_1 (J_n^h)}{\misura_1(J) h} h \misura_1(J) \\
& \geq (1-\varepsilon_0)^5 h \misura(J) \prod_{i=1}^d w_i'-w_i =  (1-\varepsilon_0)^5 \misura(J) \misura(Q).
\end{split}
\end{equation*}

\end{document}